\documentclass[10pt,leqno]{article} %%%%leqno es para numeraci\'{o}n de
   \usepackage[centertags]{amsmath}
   \usepackage{amsfonts}
   \usepackage{amsmath}
   \usepackage{amssymb}
   \usepackage{amsthm}
   \usepackage{newlfont}
   \usepackage[latin1]{inputenc}%%%%para que reconozca las tildes.
\allowdisplaybreaks[3]
\theoremstyle{plain}
\newtheorem{theorem}{Theorem}[section]
\newtheorem{lemma}[theorem]{Lemma}

\newtheorem{corollary}[theorem]{Corollary}

\theoremstyle{definition}
\newtheorem{definition}[theorem]{Definition}

\theoremstyle{remark}

\newtheorem*{remark*}{Remark}

\numberwithin{equation}{section}

\newcommand\D{{\mathcal D}}

\newcommand\CC{{\mathbb C}}

\newcommand\RR{{\mathbb R}}

\newcommand\NN{{\mathbb N}}
\newcommand\PP{{\mathbb P}}
\newcommand\Aa{{\mathbb A}}

\newcommand\X{{\Theta}}
\newcommand\x{{\theta}}

\newcommand\pp{\mbox{$\mathfrak{p}_{F}$}}

\newcommand\F{{\mathrm{U}}}

\newcommand\Sh{\mbox{\Large $\mathfrak {s}$}}

   \title{Exceptional Charlier and Hermite orthogonal polynomials
  \footnote{Partially supported by MTM2012-36732-C03-03 (Ministerio de Economía y Competitividad),
FQM-262, FQM-4643, FQM-7276 (Junta de Andalucía) and Feder Funds (European
Union).}}
   \author{Antonio J. Dur\'{a}n\\
     \footnotesize
        \  Departamento de An\'{a}lisis Matem\'{a}tico.
       Universidad de Sevilla \\
       \footnotesize Apdo (P. O. BOX) 1160. 41080 Sevilla. Spain.
   duran@us.es \\
          \ \ }
   \date{}
   
\begin{document}
   \maketitle

\bigskip

\begin{abstract}
Using Casorati determinants of Charlier polynomials $(c_n^a)_n$, we construct for each finite set $F$ of positive integers a sequence of polynomials
$c_n^{F}$, $n\in \sigma _F$, which are eigenfunctions of a second order difference operator, where $\sigma _F$ is certain infinite set of nonnegative integers, $\sigma _F \varsubsetneq \NN$.  For suitable finite sets $F$ (we call them admissible sets), we prove that the polynomials $c_n^{F}$, $n\in \sigma _F$, are actually exceptional Charlier polynomials; that is, in addition, they are orthogonal and complete with respect to a positive measure. By passing to the limit, we transform the Casorati determinant of Charlier polynomials into a Wronskian determinant of Hermite polynomials. For admissible sets, these Wronskian determinants turn out to be exceptional Hermite polynomials.
\end{abstract}

\section{Introduction}
Exceptional orthogonal polynomials $p_n$, $n\in X\varsubsetneq \NN$, are complete orthogonal polynomial systems with respect to a positive measure which in addition
are eigenfunctions of a second order differential operator. They extend the  classical families of Hermite, Laguerre and Jacobi. The last few years have seen a great deal of activity in the area  of exceptional orthogonal polynomials (see, for instance,
\cite{DEK}, \cite{GUKM1}, \cite{GUKM2} (where the adjective \textrm{exceptional} for this topic was introduced), \cite{GUKM3}, \cite{GUKM4}, \cite{GUGM}, \cite{MR}, \cite{OS0}, \cite{OS}, \cite{Qu}, \cite{STZ}, \cite{Ta} and the references therein).

The most apparent difference between classical orthogonal polynomials and exceptional orthogonal polynomials
is that the exceptional families have gaps in their degrees, in the
sense that not all degrees are present in the sequence of polynomials (as it happens with the classical families) although they form a complete orthonormal set of the underlying $L^2$ space defined by the orthogonalizing positive measure. This
means in particular that they are not covered by the hypotheses of Bochner's classification theorem \cite{B}.
Exceptional orthogonal polynomials have been applied to shape-invariant potentials \cite{Qu},
supersymmetric transformations \cite{GUKM3}, to discrete quantum mechanics \cite{OS}, mass-dependent potentials \cite{MR}, and to quasi-exact solvability \cite{Ta}.

In the same way, exceptional discrete orthogonal polynomials are complete orthogonal polynomial systems with respect to a positive measure which in addition are eigenfunction of a second order difference operator, extending the discrete classical families of Charlier, Meixner, Krawtchouk and Hahn.
As far as the author knows the only known example of what  can be called exceptional Charlier polynomials appeared in \cite{YZ} (1999). If orthogonal discrete polynomials on nonuniform lattices and orthogonal $q$-polynomials are considered, then one should add \cite{OS,OS2,OS4,OS5} where exceptional Wilson, Racah, Askey-Wilson and $q$-Racah polynomials are considered.

The purpose of this paper (and the forthcoming ones) is to introduce a systematic way of constructing exceptional discrete orthogonal polynomials using the concept of dual families of polynomials (see \cite{Leo}). One can then also construct examples of exceptional orthogonal polynomials by taking limits in some of the parameters in the same way as one goes from classical discrete polynomials to classical polynomials in the Askey tableau.

\begin{definition}\label{dfp}
Given two sets of nonnegative integers $U,V\subset \NN$, we say that the two sequences of polynomials
$(p_u)_{u\in U}$, $(q_v)_{v\in V}$ are dual if there exist a couple of sequences of numbers $(\xi_u)_{u\in U}, (\zeta_v)_{v\in V} $ such that
\begin{equation}\label{defdp}
\xi_up_u(v)=\zeta_vq_v(u), \quad u\in U, v\in V.
\end{equation}
\end{definition}

Duality has shown to be a fruitful concept regarding discrete orthogonal polynomials, and its utility will be again manifest in the exceptional discrete polynomials world. Indeed, it turns out that duality interchanges exceptional discrete orthogonal polynomials with the so-called Krall discrete orthogonal polynomials. A Krall discrete orthogonal family is a sequence of polynomials $(p_n)_{n\in \NN}$, $p_n$ of degree $n$, orthogonal with respect to a positive measure which, in addition, are also eigenfunctions of a higher order difference operator. A huge amount of families of Krall discrete orthogonal polynomials have been recently introduced by the author by mean of certain Christoffel transform of the classical discrete measures of Charlier, Meixner, Krawtchouk and Hahn (see \cite{du0}, \cite{du1}, \cite{DdI}). A Christoffel transform is a transformation which consists in multiplying a measure $\mu$ by a polynomial $r$. It has a long tradition in the context of orthogonal polynomials: it goes back a century and a half ago when
E.B. Christoffel (see \cite{Chr} and also \cite{Sz}) studied it for the particular case $r(x)=x$.

In this paper we will concentrate on exceptional Charlier and Hermite polynomials (Meixner, Krawtchouk, Hahn, Laguerre and Jacobi families will be considered in forthcoming papers).

The content of this paper is as follows. In Section 2, we include some preliminary results about symmetric operators, Christoffel transforms and finite sets of positive integers.

In Section 3, using Casorati determinants of Charlier polynomials we associate to each finite set $F$ of positive integers a sequence of polynomials which are eigenfunctions of a second order difference operator. Indeed, write the finite set $F$ of positive integers as $F=\{f_1,\cdots ,f_k\}$, $f_i<f_{i+1}$ ($k$ is then the number of elements of $F$ and $f_k$ the maximum element of $F$). We define the nonnegative integer $u_F$  by $u_F=\sum_{f\in F}f-\binom{k+1}{2}$ and the infinite set of nonnegative integers $\sigma _F$ by
$$
\sigma _F=\{u_F,u_F+1,u_F+2,\cdots \}\setminus \{u_F+f,f\in F\}.
$$
Given $a\in \RR \setminus \{0\}$, we then associate to $F$ the sequence of polynomials $c_n^{a;F}$, $n\in \sigma _F$, defined by
\begin{equation}\label{defchexi}
c_n^{a;F}(x)=\begin{vmatrix}c_{n-u_F}^a(x)&c_{n-u_F}^a(x+1)&\cdots &c_{n-u_F}^a(x+k)\\
c_{f_1}^a(x)&c_{f_1}^a(x+1)&\cdots &c_{f_1}^a(x+k)\\
\vdots&\vdots&\ddots &\vdots\\
c_{f_k}^a(x)&c_{f_k}^a(x+1)&\cdots &c_{f_k}^a(x+k) \end{vmatrix} ,
\end{equation}
where $(c_n^a)_n$ are the Charlier polynomials (see (\ref{Chpol})) orthogonal with respect to the discrete measure
$$
\rho_a =\sum_{x=0}^\infty \frac{a^x}{x!}\delta _x.
$$
Consider now the measure
\begin{equation}\label{ctmc}
\rho _{a}^{F}=(x-f_1)\cdots (x-f_k)\rho _a.
\end{equation}
Orthogonal polynomials with respect to $\rho _a^F$ are eigenfunctions of higher order difference operators (see \cite{du0} and \cite{DdI}).
It turns out that the sequence of polynomials $c_n^{a;F}$, $n\in \sigma _F$, and the sequence of orthogonal polynomials $(q_n^F)_n$ with respect to the measure $\rho _{a}^{F}$ are dual sequences (see Lemma \ref{lem3.2}). As a consequence we get that the polynomials $c_n^{a;F}$, $n\in \sigma _F$, are always eigenfunctions of a second order difference operator $D_F$ (whose coefficients are rational functions); see Theorem \ref{th3.3}. Charlier-type orthogonal polynomials considered in \cite{YZ} corresponds with the case $F=\{1,2\}$ (duality is also used in \cite{YZ}).

In Section 4, we study the most interesting case: it appears when the measure $\rho _{a}^{F}$ (\ref{ctmc}) is positive. This gives rise to the concept of admissible sets of positive integers. Split up the set $F$, $F=\bigcup _{i=1}^KY_i$, in such a way that $Y_i\cap Y_j=\emptyset $, $i\not =j$, the elements of each $Y_i$ are consecutive integers and $1+\max Y_i<\min Y_{i+1}$, $i=1,\cdots, K-1$; we then say that $F$ is admissible if each $Y_i$, $i=1,\cdots, K$,  has an even number of elements. It is straightforward to see that $F$ is admissible if and only if $\prod_{f\in F}(x-f)\ge 0$, $x\in \NN$, or in other words, (if $a>0$) the measure $\rho_a^F$ (\ref{ctmc}) is positive.
This concept of admissibility has appeared several times in the literature. Relevant to this paper because of the relationship with exceptional polynomials are \cite{Kr} and \cite{Ad} where the concept appears in connection with the zeros of certain Wronskian determinant associated with eigenfunctions of second order differential operators of the form $-d^2/dx ^2 +U$. Admissibility was also considered in \cite{KS} and \cite{YZ}.

We  prove (Theorems \ref{th4.4} and \ref{th4.5}) that if $F$ is an admissible set and $a>0$, then the polynomials $c_n^{a;F}$, $n\in \sigma _F$, are orthogonal and complete with respect to the positive measure
$$
\omega_{a;F} =\sum_{x=0}^\infty \frac{a^x}{x!\Omega ^a_F(x)\Omega ^a_F(x+1)}\delta _x,
$$
where $\Omega _F^a$ is the polynomial defined by
\begin{equation}\label{defchexii}
\Omega _F ^a(x)=\begin{vmatrix}c_{f_1}^a(x)&c_{f_1}^a(x+1)&\cdots &c_{f_1}^a(x+k-1)\\
\vdots&\vdots&\ddots &\vdots\\
c_{f_k}^a(x)&c_{f_k}^a(x+1)&\cdots &c_{f_k}^a(x+k-1) \end{vmatrix} .
\end{equation}
In particular we characterize admissible sets $F$ as those for which the Casorati determinant $\Omega ^a_F(x)$ has constant sign for $x\in \NN$ (Lemma \ref{l3.1}).

Casorati determinants like (\ref{defchexii}) for Charlier and other discrete orthogonal polynomials were considered by Karlin and Szeg\H o in \cite{KS} (see also \cite{KMc}). In particular, Karlin and Szeg\H o proved that when $F$ is admissible then $\Omega _F^a(x)$ $(a>0)$ has constant sign for $x\in \NN$.

Although it is out of the scope of this paper, we point out here that the duality transforms the higher order difference operator with respect to which the polynomials $(q_n^F)_n$ are eigenfunctions in a higher order recurrence relation for the polynomials $c_n^{a;F}$. This higher order recurrence relation has the form
$$
h(x)c_n^{a;F}(x)=\sum_{j=-u_F-k-1}^{u_F+k+1}u^{a;F}_{n,j}c_{n+j}^{a;F}(x)
$$
where $h$ is a polynomial in $x$ of degree $u_F+k+1$ satisfying $h(x)-h(x-1)=\Omega _F^a(x)$, and $a_{n,j}$, $j=-u_F-k-1,\ldots ,u_F+k+1$, are rational functions in $n$ depending on $a$ and $F$ but not on $x$.

In Section 5 and 6, we construct exceptional Hermite polynomials by taking limit (in a suitable way) in the exceptional Charlier polynomials when $a\to +\infty $. We then get (see Theorem \ref{th5.1}) that for each finite set $F$ of positive integers, the polynomials
\begin{equation}\label{defhexi}
H_n^F(x)=\begin{vmatrix}H_{n-u_F}(x)&H_{n-u_F}'(x)&\cdots &H_{n-u_F}^{(k)}(x)\\
H_{f_1}(x)&H_{f_1}'(x)&\cdots &H_{f_1}^{(k)}(x)\\
\vdots&\vdots&\ddots &\vdots\\
H_{f_k}(x)&H_{f_k}'(x)&\cdots &H_{f_k}^{(k)}(x) \end{vmatrix} ,
\end{equation}
$n\in \sigma _F$, are eigenfunctions of a second order differential operator.

When $F$ is admissible,  the Wronskian determinant $\Omega _F$  defined by
\begin{equation}\label{defhexii}
\Omega _F(x)=\begin{vmatrix}
H_{f_1}(x)&H_{f_1}'(x)&\cdots &H_{f_1}^{(k-1)}(x)\\
\vdots&\vdots&\ddots &\vdots\\
H_{f_k}(x)&H_{f_k}'(x)&\cdots &H_{f_k}^{(k-1)}(x) \end{vmatrix}
\end{equation}
does not vanish in $\RR$.  For admissible sets $F$, we then prove that the polynomials $H_n^F$, $n\in \sigma _F$, are orthogonal with respect to the positive weight
$$
\omega_{F}(x) =\frac{e^{-x^2}}{\Omega ^2_F(x)},\quad x\in \RR .
$$
Moreover, they form a complete orthogonal system in $L^2(\omega _{F})$ (see Theorem \ref{th6.3}). The exceptional Hermite family introduced
in \cite{DR} corresponds with $F=\{1,2,\cdots , 2k\}$ (for the case $k=1$ see also \cite{DEK} and \cite{CPRS}). Simultaneously with this paper, exceptional Hermite polynomials as Wronskian determinant of Hermite polynomials have been introduced and studied (using a different approach) in \cite{GUGM}.

We guess that the non vanishing property of the Wronskian determinant (\ref{defhexii})  in $\RR$ is actually true for orthogonal polynomials with respect to a positive measure. Moreover, we conjecture that this property characterizes admissible sets:

\noindent
\textsl{Conjecture.} Let $F=\{f_1,\cdots, f_k\}$ and $\mu$ be a finite set of positive integers and a positive measure with finite moments and infinitely many points in its support, respectively. Consider the monic sequence $(p_n)_n$ of orthogonal polynomials with respect to $\mu$ and write $\Omega _F^{\mu}$ for the Wronskian determinant defined by $\Omega_{F}^{\mu} (x)=|p_{f_i}^{(j-1)}(x)|_{i,j=1}^k$. Then the following conditions are equivalent.
\begin{enumerate}
\item $F$ is admissible.
\item For all positive measure $\mu$ as above the Wronskian determinant $\Omega _{F}^{\mu}(x)$ does not vanish in $\RR$.
\end{enumerate}

Wronskian determinants like (\ref{defhexii}) for orthogonal polynomials were considered by Karlin and Szeg\H o in \cite{KS} for the particular case of finite sets $F$ formed by consecutive positive integers. In particular, Karlin and Szeg\H o proved the implication (1) $\Rightarrow $ (2) when $F$ is formed by an even number of consecutive positive integers.

\bigskip

When $F$ is an admissible set and $a>0$, exceptional Charlier and Hermite polynomials $c_n^{a;F}$ and $H_n^F$, $n\in \sigma _F$, can be constructed in an alternative way.
Indeed, consider the involution $I$ in the set of all finite sets of positive integers defined by
$$
I(F)=\{1,2,\cdots, f_k\}\setminus \{f_k-f,f\in F\}.
$$
The set $I(F)$ will be denoted by $G$: $G=I(F)$. We also write $G=\{g_1,\cdots , g_m\}$ with $g_i<g_{i+1}$ so that $m$ is the number of elements of $G$ and $g_m$ the maximum element of $G$. We also need the nonnegative integer $v_F$ defined by
$$
v_F=\sum_{f\in F}f+f_k-\frac{(k-1)(k+2)}{2}.
$$
For $n\ge v_F$, we then have
\begin{equation}\label{quschi2i}
c_n^{a;F}(x)=\beta_n\begin{vmatrix}
c^a_{n-v_F}(x) & \frac{x}{a}c^a_{n-v_F}(x-1) & \cdots & \frac{(x-m+1)_m}{a^m}c^a_{n-v_F}(x-m+1) \\
c^{-a}_{g_1}(-x-1) & c^{-a}_{g_1}(-x) & \cdots &
c^{-a}_{g_1}(-x+m-1) \\
               \vdots & \vdots & \ddots & \vdots \\
               c^{-a}_{g_m}(-x-1) & \displaystyle
               c^{-a}_{g_m}(-x) & \cdots &c^{-a}_{g_m}(-x+m-1)
             \end{vmatrix},
\end{equation}
\begin{equation}\label{defhexai}
H_n^F(x)=\gamma_n\begin{vmatrix}H_{n-v_F}(x)&-H_{n-v_F-1}(x)&\cdots &(-1)^mH_{n-v_F-m}(x)\\
H_{g_1}(-ix)&H_{g_1}'(-ix)&\cdots &H_{g_1}^{(m)}(-ix)\\
\vdots&\vdots&\ddots &\vdots\\
H_{g_m}(-ix)&H_{g_m}'(-ix)&\cdots &H_{g_m}^{(m)}(-ix)\end{vmatrix} ,
\end{equation}
where $\beta_n$ and $\gamma _n$, $n\ge v_F$, are certain normalization constants (see (\ref{nc1}) and (\ref{nc2})).

We have however computational evidence that shows that both identities (\ref{quschi2i}) and (\ref{defhexai}) are true for every finite set $F$ of positive integers.

Both determinantal definitions (\ref{defchexi}) and (\ref{quschi2i}) of the polynomials $c_n^{a;F}$, $n\in \sigma _F$, automatically imply a couple of factorizations of its associated second order difference operator $D_F$ in two first order difference operators. Using these factorizations, we prove that the sequence $c_n^{a;F}$, $n\in \sigma _F$, and the operator $D_F$ can be constructed in two different ways using Darboux transforms (see Definition \ref{dxt}). If we consider (\ref{defchexi}) the Darboux transform uses the sequence $c_n^{a,F_{\{ k\}}}$, $n\in \sigma _{F_{\{ k\}}}$, where $F_{\{ k\}}=\{f_1,\cdots , f_{k-1}\}$. On the other hand, if we consider (\ref{quschi2i}) the Darboux transform uses the sequence $c_n^{a;F_{\Downarrow}}$, $n\in \sigma _{F_{\Downarrow}}$, where
$$
F_{\Downarrow}=\begin{cases} \emptyset,& \mbox{if $F=\{1,2,\cdots , k\}$,}\\
\{f_{s_F}-s_F,\cdots , f_k-s_F\},& \mbox{if $F\not =\{1,2,\cdots , k\}$},
\end{cases}
$$
and for $F\not =\{1,2,\cdots , k\}$, we write $s_F=\min \{s\ge 1: s<f_s\}$. The second factorization seems to be more interesting because the operator $F\to F_{\Downarrow}$ preserves the admissibility of the set $F$. The same happens with the determinantal definitions of the exceptional Hermite polynomials $H_n^F$ (\ref{defhexi}) and (\ref{defhexai}). This fact agrees with the G\'omez-Ullate-Kamran-Milson conjecture and its corresponding discrete version (see \cite{GUKM5}): exceptional and exceptional discrete orthogonal polynomials can be obtained by applying a sequence of Darboux transforms to a classical or classical discrete orthogonal family, respectively.

We finish this Introduction by pointing out that there is a very nice invariant property of the polynomial $\Omega _F ^a$ (\ref{defchexii}) underlying the fact that the polynomials $c_n^{a;F}$, $n\in \sigma _F$, admit both determinantal definitions (\ref{defchexi}) and (\ref{quschi2i}) (see \cite{du2}, \cite{du3} and \cite{CD}): except for a sign, $\Omega_F^a$ remains invariant if we change $F$ to $G=I(F)$, $x$ to $-x$ and $a$ to $-a$; that is
$$
\Omega _F^a(x)=(-1)^{k+u_F}\Omega ^{-a}_G(-x).
$$
This invariant property gives rise to the corresponding one for the Wronskian determinant (\ref{defhexii}) (see (\ref{izah})).

\section{Preliminaries}
Let $\mu $ be a Borel measure (positive or not) on the real line. The $n$-th moment of $\mu $ is defined by
$\int _\RR t^nd\mu (t)$. When $\mu$ has finite moments for any $n\in \NN$, we can associate it a bilinear form defined in the linear space of polynomials by
\begin{equation}\label{bf}
\langle p, q\rangle =\int pqd\mu.
\end{equation}
Given an infinite set $X$ of nonnegative integers, we say that the polynomials $p_n$, $n\in X$, $p_n$ of degree $n$, are orthogonal with respect to $\mu$ if they
are orthogonal with respect to the bilinear form defined by $\mu$; that is, if they satisfy
$$
\int p_np_md\mu =0, \quad n\not = m, \quad n,m \in X.
$$
When $X=\NN$ and the degree of $p_n$ is $n$, $n\ge 0$, we get the usual definition of orthogonal polynomials with respect to a measure.
When $X=\NN$, orthogonal polynomials with respect to a measure are unique up to multiplication by non null constant. Let us remark  that this property is not true when $X\not =\NN$.
Positive measures $\mu $ with finite moments of any order and infinitely many points in its support has always a sequence of orthogonal polynomials $(p_n)_{n\in\NN }$, $p_n$ of degree $n$ (it is enough to apply the Gram-Smith orthogonalizing process to $1, x, x^2, \ldots$); in this case
the orthogonal polynomials have positive norm: $\langle p_n,p_n\rangle>0$. Moreover, given a sequence of orthogonal polynomials $(p_n)_{n\in \NN}$ with respect to a measure $\mu$ (positive or not) the bilinear form (\ref{bf}) can be represented by a positive measure if and only if $\langle p_n,p_n \rangle > 0$, $n\ge 0$.

When $X=\NN$, Favard's Theorem establishes that a sequence $(p_n)_{n\in \NN}$ of polynomials, $p_n$ of degree $n$, is orthogonal (with non null norm) with respect to a measure if and only if it satisfies
a three term recurrence relation of the form ($p_{-1}=0$)
$$
xp_n(x)=a_np_{n+1}(x)+b_np_n(x)+c_np_{n-1}(x), \quad n\ge 0,
$$
where $(a_n)_{n\in \NN}$, $(b_n)_{n\in \NN}$ and $(c_n)_{n\in \NN}$ are sequences of real numbers with $a_{n-1}c_n\not =0$, $n\ge 1$. If, in addition, $a_{n-1}c_n>0$, $n\ge 1$,
then the polynomials $(p_n)_{n\in \NN}$ are orthogonal with respect to a positive measure with infinitely many points in its support, and conversely.
Again, Favard's Theorem is not true for a sequence of orthogonal polynomials $(p_n)_{n\in X}$ when $X\not =\NN$.

We will also need the following three lemmas.
The first one is the Sylvester's determinant identity (for the proof and a more general formulation of the Sylvester's identity see \cite{Gant}, p. 32).

\bigskip
\begin{lemma}\label{lemS}
For a square matrix $M=(m_{i,j})_{i,j=1}^k$,  and for each $1\le i, j\le k$, denote by $M_i^j$ the square matrix that results from $M$ by deleting the $i$-th row and the $j$-th column. Similarly, for $1\le i, j, p,q\le k$ denote by $M_{i,j}^{p,q}$ the square matrix that results from $M$ by deleting the $i$-th and $j$-th rows and the $p$-th and $q$-th columns.
The Sylvester's determinant identity establishes that for $i_0,i_1, j_0,j_1$ with $1\le i_0<i_1\le k$ and $1\le j_0<j_1\le k$, then
$$
\det(M) \det(M_{i_0,i_1}^{j_0,j_1}) = \det(M_{i_0}^{j_0})\det(M_{i_1}^{j_1}) - \det(M_{i_0}^{j_1}) \det(M_{i_1}^{j_0}).
$$
\end{lemma}

The second and third lemmas establish some (more or less straightforward) technical properties about second order difference operators.

\bigskip
\begin{lemma}\label{lemdes}
Write $A$, $B$ and $D$ for the following  two first order and a second order difference operators
$$
A=a_0\Sh_0+a_1\Sh_1,\quad B=b_{-1}\Sh_{-1}+b_0\Sh_0,\quad
D=f_{-1}\Sh_{-1}+f_0\Sh_0+f_1\Sh_1,
$$
where $\Sh_l$ denotes the shift operator $\Sh_l(p)(x)=p(x+l)$.
Then $D=BA$ if and only if
$$
b_{-1}(x)=\frac{f_{-1}(x)}{a_0(x-1)},\quad b_0(x)=\frac{f_{1}(x)}{a_1(x)},\quad f_0(x)=\frac{f_{-1}(x)a_1(x-1)}{a_0(x-1)}+\frac{f_{1}(x)a_0(x)}{a_1(x)}.
$$
On the other hand, $D=AB$ if and only if
$$
a_{0}(x)=\frac{f_{-1}(x)}{b_{-1}(x)},\quad a_1(x)=\frac{f_{1}(x)}{b_0(x+1)},\quad f_0(x)=\frac{f_{-1}(x)b_0(x)}{b_{-1}(x)}+\frac{f_{1}(x)b_{-1}(x+1)}{b_0(x+1)}.
$$
\end{lemma}

\bigskip
\begin{lemma}\label{lemigop}
Let $D$ and $\tilde D$ be two second order difference operators with rational coefficients. Assume that there exist polynomials $p_1$, $p_2$ and $p_3$ with degrees $d_1$, $d_2$ and $d_3$, respectively, such that $D(p_i)=\tilde D(p_i)$, $i=1,2,3$.
If $d_i>0$ and $d_i\not =d_j$, $i\not =j$, then $D=\tilde D$.
\end{lemma}

\begin{proof}
Since $d_i\not =d_j$, $i\not =j$, we deduce from Lemma 3.4 of \cite{DdI} that the polynomial
$$
P(x)=\begin{vmatrix}p_1(x+1)&p_1(x)&p_1(x-1)\\
p_2(x+1)&p_2(x)&p_2(x-1)\\
p_3(x+1)&p_3(x)&p_3(x-1)
\end{vmatrix}
$$
has degree $d=d_1+d_2+d_3-3>0$. Hence $P(x)\not =0$ for $x\not \in X_P$, where $X_P$ is formed by at most $d$ complex numbers.
Given any three rational functions $g_1,g_2,g_3$, write $Y$ for the set formed by their poles.
Then, for each $x\not \in X_P\cup Y$, the linear system of equations
$D(p_i)(x)=g_i(x)$ defines uniquely the value at $x$ of the coefficients of the second order difference operator $D$. Since $D(p_i)=\tilde D(p_i)$, $i=1,2,3$, we can conclude that $D=\tilde D$ since their coefficients are equal.
\end{proof}

Given a finite set of numbers $F=\{f_1,\cdots, f_k\}$ we denote by $V_F$ the Vandermonde determinant defined by
\begin{align}\label{defvdm}
V_F=\prod_{1=i<j=k}(f_j-f_i).
\end{align}

\subsection{Symmetric operators}

Consider a  measure $\mu$ with finite moments of any order (so that we can integrate polynomials with respect to $\mu $). Let $\Aa $ be a linear subspace of the linear space of polynomials $\PP$.
We say that a linear operator $T:\Aa\to \PP$ is symmetric with respect to the pair $(\mu , \Aa)$ if
$\langle T(p),q\rangle _\mu =\langle p,T(q)\rangle _ \mu$
for all polynomials $p, q \in \Aa$, where the bilinear form $\langle \cdot, \cdot \rangle _\mu $ is defined by (\ref{bf}). The following Lemma is then straightforward.

\begin{lemma}\label{lsyo} Let $T$ be a symmetric operator with respect to the pair $(\mu ,\Aa)$. Assume we have polynomials $r_n\in \Aa$, $n\in X\subset \NN$, which are
eigenfunctions for the operator $T$ with different eigenvalues, that is,
$T(r_n)=\lambda_n r_n$, $n\in X$, and $\lambda _n\not =\lambda_m$, $n\not =m$. Then the polynomials $r_n$, $n\in X$, are orthogonal with respect to $\mu$.
\end{lemma}

When $\mu$ is a discrete measure, the symmetry of a finite order difference operator $D=\sum_{l=-r}^rh_l\Sh _l$ with respect to a pair $(\mu, \Aa)$ can be guaranteed by a finite set of difference equations together with certain boundary conditions. The proof follows as that of Theorem 3.2 in \cite{du0}
and it is omitted.

\begin{lemma}\label{tcsd} Let $\mu$  be a  discrete measure supported on a countable set $X$, $X\subset \RR$.
Consider a finite order difference operator $T:\Aa \to \PP$ of the form $T=\sum _{l=-r}^rh_l\Sh_l $, where $\Aa $ is a linear subspace of the linear space of polynomials $\PP$. Assume that the
measure $\mu $ and the coefficients $h_l$, $l=-r,\cdots , r$, of
$T$ satisfy the  difference equations
\begin{equation}\label{desm}
h_l(x-l)\mu (x-l)=h_{-l}(x)\mu (x),\quad \mbox{for $x\in (l+X)\cap X$ and $l=1, \cdots, r$,}
\end{equation}
and the boundary conditions
\begin{align}\label{bc1}
h_l(x-l)&=0, \quad
\mbox{for $x\in (l+X)\setminus X$ and $l=1, \cdots, r$,}\\
\label{bc2}
h_{-l}(x)&=0, \quad
\mbox{for $x\in X\setminus (l+X)$ and $l=1, \cdots, r$.}
\end{align}
Then $T$ is symmetric with respect to the pair $(\mu, \Aa)$.
(Let us remind that for a set of numbers $A$ and a number $b$, we denote by $b+A$ the set $b+A=\{ b+a: a\in A\}$).
\end{lemma}

On the other hand, when $\mu$ has a smooth density with respect to the Lebesgue measure (i. e. $d\mu =f(x)dx$), the symmetry of a second order differential operator with respect to a pair $(\mu, \Aa)$ can be guaranteed by the usual Pearson equation. The proof follows by performing an integration by parts
and it is omitted.

\begin{lemma}\label{tcsd2} Let $\mu$  be a measure having a positive $\mathcal{C}^2(I)$ density $f$ with respect to the Lebesgue measure in an interval $I\subset \RR $. Consider a second order differential operator $T:\Aa \to \PP$ of the form $T=a_2(x)\partial ^2 +a_1(x)\partial +a_0(x)$, where $\Aa $ is a linear subspace of the linear space of polynomials $\PP$, $\partial =d/dx$ and $a_2$ and $a_1$  are $\mathcal{C}^1(I)$ functions.
Assume that  $f $ and the coefficients $a_2$, $a_1$  of
$T$ satisfy the Pearson equation
$$
(a_2(x)f(x))'=a_1(x)f(x), \quad x\in I,
$$
and the boundary conditions that the limit of the functions $x^na_2(x)f(x)$ and $x^n(a_2(x)f(x))'$ vanish at the endpoints of $I$ for $n\ge 0$.
Then $T$ is symmetric with respect to the pair $(\mu, \Aa)$.
\end{lemma}

\subsection{Christoffel transform}\label{secChr}
Let $\mu$ be a measure (positive or not) and assume that $\mu$ has a sequence of orthogonal polynomials
$(p_n)_{n\in \NN}$, $p_n$ with degree $n$ and $\langle p_n,p_n\rangle \not =0$ (as we mentioned above, that always happens if $\mu$ is positive, with finite moments and infinitely many points in its support).
Favard's theorem implies that the sequence of polynomials $(p_n)_n$ satisfies the three term recurrence relation ($p_{-1}=0$)
\begin{equation}\label{rrpn}
xp_n(x)=a_n^Pp_{n+1}(x)+b_n^Pp_n(x)+c_n^Pp_{n-1}(x).
\end{equation}

Given a finite set $F$ of real numbers, $F=\{f_1,\cdots , f_k\}$, $f_i<f_{i+1}$, we write $\Phi_n$, $n\ge 0$, for the $k\times k$ determinant
\begin{equation}\label{defph}
\Phi_n=\vert p_{n+j-1}(f_i)\vert _{i,j=1,\cdots , k}.
\end{equation}
Notice that $\Phi_n$, $n\ge 0$, depends on both, the finite set $F$ and the measure $\mu$. In order to stress this dependence, we sometimes write in this Section $\Phi_n^{\mu, F}$ for $\Phi_n$.

Along this Section we assume that the set $\X_\mu^F=\{ n\in \NN :\Phi_n^{\mu,F}=0\}$ is finite. We denote $\x_\mu ^F=\max \X_\mu ^F$. If
$\X_\mu ^F=\emptyset$ we take $\x_\mu ^F=-1$.

The Christoffel transform of $\mu$ associated to the annihilator polynomial $\pp$ of $F$,
$$
\pp (x)=(x-f_1)\cdots (x-f_k),
$$
is the measure defined by $ \mu_F =\pp \mu$.

Orthogonal polynomials with respect to $\mu_F$ can be constructed by means of the formula
\begin{equation}\label{mata00}
q_n(x)=\frac{1}{\pp (x)}\det \begin{pmatrix}p_n(x)&p_{n+1}(x)&\cdots &p_{n+k}(x)\\
p_n(f_1)&p_{n+1}(f_1)&\cdots &p_{n+k}(f_1)\\
\vdots&\vdots&\ddots &\vdots\\
p_n(f_k)&p_{n+1}(f_k)&\cdots &p_{n+k}(f_k) \end{pmatrix}.
\end{equation}
Notice that the degree of $q_n$ is equal to $n$ if and only if $n\not\in \X_\mu ^F$. In that case the leading coefficient $\lambda^Q_n$ of $q_n$ is
equal to $(-1)^k\lambda^P_{n+k}\Phi_n$, where $\lambda ^P_n$ denotes the leading coefficient of $p_n$.

The next Lemma follows easily using \cite{Sz}, Th. 2.5.

\begin{lemma}\label{sze}
The measure $\mu_F$ has a sequence $(q_n)_{n=0}^\infty $, $q_n$ of degree $n$, of orthogonal polynomials if and only if $\X_\mu ^F=\emptyset$.
In that case, an orthogonal polynomial of degree $n$ with respect to $\mu _F$ is given by (\ref{mata00}) and also $\langle q_n,q_n\rangle _{\mu _F}\not =0$, $n\ge 0$. If $\X_\mu \not =\emptyset$, the polynomial $q_n$ (\ref{mata00}) has still degree $n$ for $n\not \in \X_\mu^F$, and satisfies $\langle q_n,r\rangle_{\mu _F}=0$ for all polynomial $r$ with degree less than $n$ and $\langle q_n,q_n\rangle _{\mu _F}\not =0$.
\end{lemma}

The three term recurrence relation for the polynomials $(q_n)_n$ can be derived from the corresponding recurrence relation for the polynomials $(p_n)_n$ (\ref{rrpn}). In addition to the determinant $\Phi_n$ (\ref{defph}), $n\ge 0$, we also consider the $k\times k$ determinant
\begin{equation}\label{defps}
\Psi_n=\begin{vmatrix}
p_n(f_1)&p_{n+1}(f_1)&\cdots &p_{n+k-2}(f_1)&p_{n+k}(f_1)\\
\vdots&\vdots&\ddots &\vdots&\vdots\\
p_n(f_k)&p_{n+1}(f_k)&\cdots &p_{n+k-2}(f_k)&p_{n+k}(f_k) \end{vmatrix}.
\end{equation}

\begin{lemma}\label{lemmc}
For $n> \x_\mu ^F+1$, the polynomials $q_n$ (\ref{mata00}) satisfy the three term recurrence relation
\begin{equation}\label{rrvqn}
xq_n(x)=a_n^Qq_{n+1}(x)+b_n^Qq_n(x)+c_n^Qq_{n-1}(x),
\end{equation}
where
\begin{align*}
a^Q_n&=a_n^P\frac{\lambda^P_{n+1}\lambda^P_{n+k}}{\lambda^P_{n}\lambda^P_{n+k+1}}\frac{\Phi _n}{\Phi_{n+1}},\\
b^Q_n&= b_{n+k}^P+\frac{\lambda^P_{n+k}}{\lambda^P_{n+k+1}}\frac{\Psi _{n+1}}{\Phi_{n+1}}-\frac{\lambda^P_{n+k-1}}{\lambda^P_{n+k}}\frac{\Psi _n}{\Phi_{n}},\\
c^Q_n&=c_n^P\frac{\Phi _{n+1}}{\Phi_{n}}.
\end{align*}
Moreover,
\begin{equation}\label{n2q}
\langle q_n,q_n\rangle _{\mu_F}=(-1)^k\frac{\lambda^P_{n+k}}{\lambda^P_{n}}\Phi_n\Phi_{n+1}\langle p_n,p_n\rangle _{\mu},\quad n> \x_\mu^F+1.
\end{equation}
If $\X_\mu ^F=\emptyset$, then (\ref{rrvqn}) and (\ref{n2q}) hold for $n\ge 0$, with initial condition $q_{-1}=0$.
\end{lemma}

\begin{proof}
We can assume that $p_n$ are monic (that is, $\lambda^P_n=1$). Write $\hat q_n(x)=q_n(x)/\lambda_n^Q$, $n> \x_\mu^F +1$.
It is then enough to prove that
$$
x\hat q_n(x)=a_n\hat q_{n+1}(x)+b_n\hat q_n(x)+c_n\hat q_{n-1}(x)
$$
with
\begin{align}\label{rrma}
a_n&=1,\\\label{rrmb}
b_n&=b_{n+k}^P+\frac{\Psi _{n+1}}{\Phi_{n+1}}-\frac{\Psi _n}{\Phi_{n}},\\\label{rrmc}
c_n&=c_n^P\frac{\Phi _{n-1}\Phi_{n+1}}{\Phi^2_{n}},\\\label{rrmq}
\langle q_n,q_n\rangle _{\mu_F} &=\frac{\langle p_{n},p_{n}\rangle _{\mu}\Phi_{n+1}}{\Phi_{n}}.
\end{align}
We write $u_n(x)=x^n$ for $n\le \x_\mu^F+1$, and $u_n(x)=\hat q_n (x)$ for $n>\x_\mu^F+1$. Then the polynomials $u_n$, $n\ge 0$, form a basis of $\PP$. From the previous lemma,
we also have for $n>\x ^F_\mu$ that $\langle \hat q_n,u_j\rangle _{\mu_F}=0$, $j=0,\cdots , n-1$.
Taking this into account, it is easy to deduce that the polynomials $\hat q_n$, $n> \x_\mu^F+1$, satisfy a three term recurrence relation
$$
x\hat q_n(x)=a_n\hat q_{n+1}(x)+b_n\hat q_n(x)+c_n\hat q_{n-1}(x).
$$
Since they are monic, we straightforwardly have $a_n=1$, that is, (\ref{rrma}). We compute $c_n$ as
$$
c_n=\frac{\langle x\hat q_n,\hat q_{n-1}\rangle _{\mu_F}}{\langle \hat q_{n-1},\hat q_{n-1}\rangle _{\mu_F}}=
\frac{\langle x\hat q_n, p_{n-1}\rangle _{\mu_F}}{\langle \hat q_{n-1},\hat q_{n-1}\rangle _{\mu_F}}.
$$
Using (\ref{mata00}), we get
\begin{align*}
\langle x\hat q_n,p_{n-1}\rangle _{\mu_F}&=\frac{(-1)^k}{\Phi_n}\begin{vmatrix}\langle xp_n,p_{n-1}\rangle _{\mu}&\langle xp_{n+1},p_{n-1}\rangle _{\mu}
&\cdots &\langle xp_{n+k},p_{n-1}\rangle _{\mu}\\
p_n(f_1)&p_{n+1}(f_1)&\cdots &p_{n+k}(f_1)\\
\vdots&\vdots&\ddots &\vdots\\
p_n(f_k)&p_{n+1}(f_k)&\cdots &p_{n+k}(f_k) \end{vmatrix}\\
&=\frac{(-1)^k}{\Phi_n}\begin{vmatrix}c^P_n\langle p_{n-1},p_{n-1}\rangle _{\mu}&0
&\cdots &0\\
p_n(f_1)&p_{n+1}(f_1)&\cdots &p_{n+k}(f_1)\\
\vdots&\vdots&\ddots &\vdots\\
p_n(f_k)&p_{n+1}(f_k)&\cdots &p_{n+k}(f_k) \end{vmatrix}\\
&=(-1)^k\frac{c^P_n\langle p_{n-1},p_{n-1}\rangle _{\mu}\Phi_{n+1}}{\Phi_n}.
\end{align*}
In a similar way, one finds that
$$
\langle \hat q_{n-1},\hat q_{n-1}\rangle _{\mu_F}=(-1)^k\frac{\langle p_{n-1},p_{n-1}\rangle _{\mu}\Phi_{n}}{\Phi_{n-1}}.
$$
(\ref{rrmc}) and (\ref{rrmq}) can now be easily deduced. (\ref{rrmb}) can be proved analogously.

\end{proof}

For $\X_\mu ^F=\emptyset$, the previous lemma has already appeared in the literature (see for instance \cite{YZ}).

\subsection{Finite set of positive integers}\label{sfspi}
From now on, $F$ will denote a finite set of positive integers. We will write $F=\{ f_1,\cdots , f_k\}$, with $f_i<f_{i+1}$. Hence $k$ is the number of elements of $F$ and $f_k$ is the maximum element of $F$.

We associate to $F$ the nonnegative integers $u_F$ and $v_F$ and the infinite set of nonnegative integers $\sigma_F$ defined by
\begin{align}\label{defuf}
u_F&=\sum_{f\in F}f-\binom{k+1}{2},\\\label{defvf}
v_F&=\sum_{f\in F}f+f_k-\frac{(k-1)(k+2)}{2},\\\label{defsf}
\sigma _F&=\{u_F,u_F+1,u_F+2,\cdots \}\setminus \{u_F+f,f\in F\}.
\end{align}
The infinite set $\sigma_F$ will be the set of indices for the exceptional Charlier or Hermite polynomials associated to $F$.

Notice that $v_F=u_F+f_k+1$; hence $\{v_F,v_F+1,v_F+2,\cdots \}\subset \sigma_F$. Notice also that
$u_F$ is an increasing function with respect to the inclusion order, that is, if $F\subset \tilde F$ then $u_F\le u_{\tilde F}$.

Consider the set $\Upsilon$  formed by all finite sets of positive  integers:
\begin{align*}
\Upsilon=\{F:\mbox{$F$ is a finite set of positive integers}\} .
\end{align*}
We consider the involution $I$ in $\Upsilon$ defined by
\begin{align}\label{dinv}
I(F)=\{1,2,\cdots, f_k\}\setminus \{f_k-f,f\in F\}.
\end{align}
The definition of $I$ implies that $I^2=Id$.

For the involution $I$, the bigger the holes in $F$ (with respect to the set $\{1,2,\cdots , f_k\}$), the bigger the involuted set $I(F)$.
Here it is a couple of examples
$$
I(\{ 1,2,3,\cdots ,k\})=\{ k\},\quad \quad I(\{1, k\})=\{ 1,2,\cdots, k-2, k\}.
$$
The set $I(F)$ will be denoted by $G$: $G=I(F)$. We also write $G=\{g_1,\cdots , g_m\}$ with $g_i<g_{i+1}$ so that $m$ is the number of elements of $G$ and $g_m$ the maximum element of $G$. Notice that
$$
f_k=g_m,\quad m=f_k-k+1.
$$
We also define the number $s_F$ by
\begin{equation}\label{defs0}
s_F=\begin{cases} 1,& \mbox{if $F=\emptyset$},\\
k+1,&\mbox{if $F=\{1,2,\cdots , k\}$},\\
\min \{s\ge 1:s<f_s\}, & \mbox{if $F\not =\{1,2,\cdots k\}$}.
\end{cases}
\end{equation}

For $1\le i\le k$, we denote by $F_{\{ i\}}$ and $F_{\Downarrow}$ the finite sets of positive integers defined by
\begin{align}\label{deff0}
F_{\{ i\} }&=F\setminus \{f_i\},\\\label{deff1}
F_{\Downarrow}&=\begin{cases} \emptyset,& \mbox{if $F=\{1,2,\cdots , k\}$,}\\
\{f_{s_F}-s_F,\cdots , f_k-s_F\},& \mbox{if $F\not =\{1,2,\cdots , k\}$}.
\end{cases}
\end{align}
The following relation is straightforward from (\ref{dinv}), (\ref{deff0}) and (\ref{deff1}):
\begin{equation}\label{ref0f1}
F_{\Downarrow}=I\left(G _{\{ m\} }\right)
\end{equation}
(where as indicated above $G=I(F)$ and $m$ is the number of elements of $G$).

\subsection{Charlier and Hermite polynomials}
We include here basic definitions and facts about Charlier and Hermite polynomials, which we will need in the following sections.

For $a\neq0$, we write $(c_n^a)_n$ for the sequence of Charlier polynomials (the next formulas can be found in \cite{Ch}, pp. 170-1; see also \cite{KLS}, pp., 247-9 or \cite{NSU}, ch. 2) defined by
\begin{equation}\label{Chpol}
    c_n^a(x)=\frac{1}{n!}\sum_{j=0}^n(-a)^{n-j}\binom{n}{j}\binom{x}{j}j!.
\end{equation}
The Charlier polynomials are orthogonal with respect to the measure
\begin{equation}\label{Chw}
    \rho_a=\sum_{x=0}^{\infty}\frac{a^x}{x!}\delta_x,\quad a\neq0,
\end{equation}
which is positive only when $a>0$ and then
\begin{equation}\label{norCh}
\langle c_n^a,c_n^a\rangle=\frac{a^n}{n!}e^a.
\end{equation}
The three-term recurrence formula for $(c_n^a)_n$ is ($c_{-1}^a=0$)
\begin{equation}\label{Chttrr}
   xc_n^a=(n+1)c_{n+1}^a+(n+a)c_n^a+ac_{n-1}^a,\quad n\geq0.
\end{equation}
They are eigenfunctions of the following second-order difference operator
\begin{equation}\label{Chdeq}
   D_a=-x\Sh_{-1}+(x+a)\Sh_0-a\Sh_1,\quad D_a(c_n^a)=nc_n^a,\quad n\geq0,
\end{equation}
where $\Sh_j(f)=f(x+j)$. They also satisfy
\begin{equation}\label{Chlad}
   \Delta(c_n^a)=c_{n-1}^a, \quad \frac{d}{da}(c_n^a)=-c_{n-1}^a,
\end{equation}
and the duality
\begin{equation}\label{Chdua}
   (-1)^ma^mn!c_n^a(m)=(-1)^na^nm!c_{m}^a(n), \quad n,m\ge 0.
\end{equation}

We write $(H_n)_n$ for the sequence of Hermite polynomials (the next formulas can be found in \cite{Ch}, Ch. V; see also \cite{KLS}, pp, 250-3) defined by
\begin{equation}\label{Hpol}
 H_n(x)=n!\sum_{j=0}^{[n/2]}\frac{(-1)^j(2x)^{n-2j}}{j!(n-2j)!}.
\end{equation}
The Hermite polynomials are orthogonal with respect to the weight function $e^{-x^2}$, $x\in \RR$.
They are eigenfunctions of the following second-order differential operator
\begin{equation}\label{Hdeq}
   D=\partial ^2-2x\partial,\quad D(H_n)=-2nH_n,\quad n\geq0,
\end{equation}
where $\partial =d/dx$.

They also satisfy $H_n'(x)=2nH_{n-1}(x)$.

One can obtain Hermite polynomials from Charlier polynomials using the limit
\begin{equation}\label{blchh}
\lim_{a\to \infty}\left(\frac{2}{a}\right)^{n/2}c_n^a(\sqrt {2a}x+a)=\frac{1}{n!}H_n(x)
\end{equation}
see \cite{KLS}, p. 249 (take into account that if we write $(C_n^a)_n$ for the polynomials defined by (9.14.1) in \cite{KLS}, p. 247,
then $c_n^a=(-a)^nC_n^a/n!$). The previous limit is uniform in compact sets of $\CC$.

\section{Constructing polynomials which are eigenfunctions of second order difference operators}
As in Section \ref{sfspi}, $F$ will denote a finite set of positive integers. We will write $F=\{ f_1,\cdots , f_k\}$, with $f_i<f_{i+1}$. Hence $k$ is the number of elements of $F$ and $f_k$ is the maximum element of $F$.

We associate to each finite set $F$ of positive integers the polynomials $c_n^{a;F}$, $n\in \sigma_F$, displayed in the following definition.
It turns out that these polynomials are always eigenfunctions of a second order difference operator with rational coefficients. We call them exceptional Charlier polynomials when, in addition, they are orthogonal and complete with respect to a positive measure (this will happen as long as the finite set $F$ is admissible; see Definition \ref{defadm} in the next Section).

\begin{definition}
For a given real number $a\not =0$ and a finite set $F$ of positive integers, we define the polynomials $c_n^{a;F}$, $n\in \sigma _F$, as
\begin{equation}\label{defchex}
c_n^{a;F}(x)=\begin{vmatrix}c_{n-u_F}^a(x)&c_{n-u_F}^a(x+1)&\cdots &c_{n-u_F}^a(x+k)\\
c_{f_1}^a(x)&c_{f_1}^a(x+1)&\cdots &c_{f_1}^a(x+k)\\
\vdots&\vdots&\ddots &\vdots\\
c_{f_k}^a(x)&c_{f_k}^a(x+1)&\cdots &c_{f_k}^a(x+k) \end{vmatrix} ,
\end{equation}
where the number $u_F$ and the infinite set of nonnegative integers $\sigma _F$ are defined by (\ref{defuf}) and (\ref{defsf}), respectively.
\end{definition}

To simplify the notation, we will sometimes write $c_n^F=c_n^{a;F}$.

Using Lemma 3.4 of \cite{DdI}, we deduce that $c_n^F$, $n\in \sigma _F$, is a polynomial of degree $n$ with leading coefficient equal to
\begin{equation}\label{lcrn}
\frac{\prod_{i=1}^k(f_i-n+u_F)}{(n-u_F)!\prod_{f\in F}f!}V_F,
\end{equation}
where $V_F$ is the Vandermonde determinant (\ref{defvdm}). With the convention that $c_n^a=0$ for $n<0$, the determinant (\ref{defchex}) defines a polynomial for any $n\ge 0$, but for $n\not \in \sigma_F$ we have $c_n^F=0$.

Combining columns in (\ref{defchex}) and taking into account the first formula in (\ref{Chlad}), we have the alternative definition
\begin{equation}\label{defchexa}
c_n^F(x)=\begin{vmatrix}c_{n-u_F}^a(x)&c_{n-u_F-1}^a(x)&\cdots &c_{n-u_F-k}^a(x)\\
c_{f_1}^a(x)&c_{f_1-1}^a(x)&\cdots &c_{f_1-k}^a(x)\\
\vdots&\vdots&\ddots &\vdots\\
c_{f_k}^a(x)&c_{f_k-1}^a(x)&\cdots &c_{f_k-k}^a(x) \end{vmatrix} .
\end{equation}

The  polynomials $c_n^F$, $n\in \sigma_F$, are strongly related by duality with the polynomials $q_n^F$, $n\ge 0$, defined by
\begin{equation}\label{defqnch}
q_n^F(x)=\frac{\begin{vmatrix}c_n^a(x-u_F)&c_{n+1}^a(x-u_F)&\cdots &c_{n+k}^a(x-u_F)\\
c_n^a(f_1)&c_{n+1}^a(f_1)&\cdots &c_{n+k}^a(f_1)\\
\vdots&\vdots&\ddots &\vdots\\
c_n^a(f_k)&c_{n+1}^a(f_k)&\cdots &c_{n+k}^a(f_k) \end{vmatrix}}{\prod_{f\in F}(x-f-u_F)} .
\end{equation}

\begin{lemma}\label{lem3.2}
If $u$ is a nonnegative integer and $v\in \sigma_F$, then
\begin{equation}\label{duaqnrn}
q_u^F(v)=\xi_u\zeta_vc_v^F(u),
\end{equation}
where
$$
\xi_u=\frac{(-a)^{(k+1)u}}{\prod_{i=0}^k(u+i)!},\quad \zeta_v=\frac{(-a)^{-v}(v-u_F)!\prod_{f\in F}f!}{\prod_{f\in F}(v-f-u_F)}.
$$
\end{lemma}

\begin{proof}
It is a straightforward consequence of the duality (\ref{Chdua}) for the Charlier polynomials.

\end{proof}

We now prove that the polynomials $c_n^F$, $n\in \sigma_F$, are eigenfunctions of a second order difference operator with rational coefficients. To establish the result in full, we need some more notations.
We denote by $\Omega _F^a (x)$ and $\Lambda _F^a(x)$ the polynomials
\begin{align}\label{defom}
\Omega _F^a(x)&=|c_{f_i}^a(x+j-1)|_{i,j=1}^k,\\
\label{deflam}
\Lambda _F^a(x)&=\begin{vmatrix}
c_{f_1}^a(x)&c_{f_1}^a(x+1)&\cdots &c_{f_1}^a(x+k-2)&c_{f_1}^a(x+k)\\
\vdots&\vdots&\ddots &\vdots\\
c_{f_k}^a(x)&c_{f_k}^a(x+1)&\cdots &c_{f_k}^a(x+k-2)&c_{f_k}^a(x+k) \end{vmatrix} .
\end{align}
To simplify the notation, we will sometimes write $\Omega ^F=\Omega_F^{a}$ and $\Lambda ^F=\Lambda_F^{a}$.

Using Lemma 3.4 of \cite{DdI} and the definition of $u_F$ (\ref{defuf}), we deduce that the degree of both $\Omega _F$ and $\Lambda_F$
is $u_F+k$. From (\ref{defchex}) and (\ref{defom}), we have
\begin{equation}\label{rrom0}
\Omega _F(x)=(-1)^{k-1}c_{f_k+u_{F_{\{ k\} }}}^{F_{\{ k\} }}(x),
\end{equation}
where the finite set of positive integers $F_{\{ k\}}$ is defined by (\ref{deff0}).

As for $c_n^F$ (see (\ref{defchexa})), we have for $\Omega_F$ the following alternative definition
\begin{equation}\label{defoma}
\Omega _F(x)=|c_{f_i-j+1}^a(x)|_{i,j=1}^k.
\end{equation}
From here and (\ref{defchexa}), it is easy to deduce that
\begin{equation}\label{rrom}
c_{u_F}^F(x)=\Omega_{F_\Downarrow }(x),
\end{equation}
where the finite set of positive integers $F_\Downarrow$ is defined by (\ref{deff1}).

A simple calculation using the third formula in (\ref{Chlad}) shows that
\begin{equation}\label{relomla}
\Lambda ^a_F(x)=k\Omega ^a_F(x)-\frac{d}{da}\Omega ^a_F(x).
\end{equation}
We also need the determinants $\Phi_n^F$ and $\Psi_n^F$, $n\ge 0$, defined by
\begin{align}\label{defphch}
\Phi^F_n&=|c_{n+j-1}^a(f_i)|_{i,j=1}^k,\\\label{defpsch}
\Psi_n^F&=\begin{vmatrix}
c_n^a(f_1)&c_{n+1}^a(f_1)&\cdots &c_{n+k-2}^a(f_1)&c_{n+k}^a(f_1)\\
\vdots&\vdots&\ddots &\vdots\\
c_n^a(f_k)&c_{n+1}^a(f_k)&\cdots &c_{n+k-2}^a(f_k)&c_{n+k}^a(f_k) \end{vmatrix}.
\end{align}
Using the duality (\ref{Chdua}), we have
\begin{align}\label{duomph}
\Omega _F(n)&=\frac{\prod_{i=0}^{k-1}(n+i)!}{(-a)^{k(n-1)-u_F}\prod_{f\in F}f!}\Phi_n^F,\\ \label{duomps}
\Lambda _F(n)&=\frac{(n+k)!\prod_{i=0}^{k-2}(n+i)!}{(-a)^{k(n-1)-u_F+1}\prod_{f\in F}f!}\Psi_n^F.
\end{align}
According to Lemma \ref{sze}, as long as $\Phi_n^F\not =0$, $n\ge 0$, the  polynomials  $q_n^F$, $n\ge 0$, are orthogonal with respect to the measure
\begin{equation}\label{mraf}
\rho _{a}^{F}=\sum _{x=u_F}^\infty \prod_{f\in F}(x-f-u_F)\frac{a^{x-u_F}}{(x-u_F)!}\delta _x.
\end{equation}
Notice that the measure $\rho_{a}^{F}$ is supported in the infinite set of nonnegative integers  $\sigma_F$ (\ref{defsf}).

\begin{theorem}\label{th3.3} Let $F$ be a finite set of positive integers. Then the polynomials $c_n^F$, $n\in \sigma _F$, (\ref{defchex})
are common eigenfunctions of the second order difference operator
\begin{equation}\label{sodochex}
D_F=h_{-1}(x)\Sh_{-1}+h_0(x)\Sh_0+h_1(x)\Sh_{1},
\end{equation}
where
\begin{align}\label{jpm1}
h_{-1}(x)&=-x\frac{\Omega_F(x+1)}{\Omega_F(x)},\\\label{jpm2}
h_0(x)&=x+k+a+u_F-a\frac{\Lambda_F(x+1)}{\Omega_F(x+1)}+a\frac{\Lambda_F(x)}{\Omega_F(x)},\\\label{jpm3}
h_1(x)&=-a\frac{\Omega_F(x)}{\Omega_F(x+1)}.
\end{align}
Moreover $D_F(c_n^F)=nc_n^F$, $n\in \sigma _F$.
\end{theorem}

\begin{proof}
Consider the set $\X_a^F$ of nonnegative integers defined by $\X_a^F=\{n\in \NN: \Phi_n^F=0\}$. Using (\ref{duomph}), we get $\X_a^F=\{x\in \NN: \Omega_F(x)=0\}$. Since $\Omega_F$ is a polynomial in $x$,
we conclude that $\X_a^F$ is finite. Define then $\x_a^F=\max \X_a^F$, with the convention that if $\X_a^F=\emptyset$ then $\x_a^F=-1$.

Write $p_n(x)=c_n^a(x-u_F)$ and $q_n(x)=q_n^F(x)$ (see (\ref{defqnch})). With the notation of Section \ref{secChr}, we have
$$
\lambda^P_n=\frac{1}{n!},\quad \lambda^Q_n=\frac{(-1)^k\Phi_n^F}{(n+k)!}.
$$
Using the three term recurrence relations (\ref{Chttrr}) for the Charlier polynomials and (\ref{rrvqn}) for $q_n$, $n>\x_a^F+1$,  we conclude after an easy calculation that for $u>\x_a^F+1$ and $v\in \RR $
\begin{equation}\label{yttr}
vq_u^F(v)=a_u^Qq_{u+1}^F(v)+b_u^Qq_u^F(v)+c_u^Qq_{u-1}^F(v),
\end{equation}
where
\begin{align}\label{anqch}
a_n^Q&=(n+k+1)\frac{\Phi _n^F}{\Phi_{n+1}^F},\\\label{bnqch}
b_n^Q&=(n+k+a+u_F)+(n+k+1)\frac{\Psi _{n+1}^F}{\Phi_{n+1}^F}-(n+k)\frac{\Psi _{n}^F}{\Phi_{n}^F},\\\label{cnqch}
c_n^Q&=a\frac{\Phi _{n+1}^F}{\Phi_{n}^F}.
\end{align}
Assume now that $v\in \sigma_F$. Then, using the dualities (\ref{duaqnrn}), (\ref{duomph}) and (\ref{duomps}), we get from (\ref{yttr}) after straightforward calculations
\begin{equation}\label{edho}
uc_v^F(u)=h_1(u)c_{v}^F(u+1)+h_0(u)c_v^F(u)+h_{-1}(u)c_v^F(u-1),
\end{equation}
for all nonnegative integers $u>\x_a^F+1$, where $h_1$, $h_0$ and $h_{-1}$ are given by (\ref{jpm1}), (\ref{jpm2}) and (\ref{jpm3}), respectively. Since $c_v^F$, $v\in \sigma_F$, are polynomials and $h_1,h_0$ and $h_{-1}$ are rational functions, we have that (\ref{edho}) holds also for all complex number $u$. In other words, the  polynomials $c_n^F$, $n\in \sigma_F$, are eigenfunctions
of the second order difference operator $D_F$ (\ref{sodochex}).

\end{proof}

The determinant which defines $\Omega_F^a$ (\ref{defom}) enjoys a very nice invariant property with respect to the involution $I$ defined by (\ref{dinv}). Indeed,
for a finite set $F=\{f_1,\cdots , f_k\}$ of positive integers, consider the involuted set $I(F)=G=\{ g_1,\cdots, g_m\}$ with $g_i<g_{i+1}$. We also need the associated functions $\tilde \Omega _F^a$ and $\tilde \Lambda _F^a$ defined by
\begin{align}\label{defomt}
\tilde \Omega _F^a(x)&=|c_{g_i}^{-a}(-x+j-1)|_{i,j=1}^m,\\\label{delamt}
\tilde \Lambda _F^a(x)&=\begin{vmatrix}
c_{g_1}^{-a}(-x)&c_{g_1}^{-a}(-x+1)&\cdots &c_{g_1}^{-a}(-x+m-2)&c_{g_1}^{-a}(-x+m)\\
\vdots&\vdots&\ddots &\vdots\\
c_{g_m}^{-a}(-x)&c_{g_m}^{-a}(-x+1)&\cdots &c_{g_m}^{-a}(-x+m-2)&c_{g_m}^{-a}(-x+m) \end{vmatrix} .
\end{align}
Using the definition of the involution $I$, we have that both $\tilde \Omega _F^a$ and $\tilde \Lambda _F^a$ are polynomials of degree $u_F+k$ (this last can be deduced using Lemma 3.4 of \cite{DdI}).

The invariant property mentioned above for  $\Omega_F^a$ (\ref{defom}) is the following: except for a sign, $\Omega_F^a$ remains invariant if we change $F$ to $G=I(F)$, $x$ to $-x$ and $a$ to $-a$. In other words, except for a sign, $\Omega_F^a$ and $\tilde \Omega_F^a$ are equal:
\begin{equation}\label{iza}
\Omega_F^a(x)=(-1)^{k+u_F}\tilde \Omega_F^a (x).
\end{equation}
For finite sets $F$ formed by consecutive positive integers this invariance was conjecture in \cite{du2} and proved in \cite{du3}. The proof for all finite set of positive integers will be included in \cite{CD}.

According to this invariant property, we can rewrite as follows the second order difference operator $D_F$ for which the polynomials $c_n^F$, $n\in \sigma _F$, are common eigenfunctions.

\begin{theorem}\label{th3.6}
Let $F$ be a finite set of positive integers. Then  the coefficients $h_{-1}$, $h_0$, $h_1$ of the operator $D_F$ (\ref{sodochex}) can be rewritten in the form
\begin{align}\label{jm1}
h_{-1}(x)&=-x\frac{\tilde \Omega_F(x+1)}{\tilde \Omega_F(x)},\\\label{jm2}
h_0(x)&=x+m+a+u_G+a\frac{\tilde \Lambda_F(x+1)}{\tilde \Omega_F(x+1)}-a\frac{\tilde \Lambda_F(x)}{\tilde \Omega_F(x)},\\\label{jm3}
h_1(x)&=-a\frac{\tilde \Omega_F(x)}{\tilde \Omega_F(x+1)}.
\end{align}
\end{theorem}

\begin{proof}
Using (\ref{iza}) and (\ref{relomla}), we straightforwardly get (\ref{jm1}), (\ref{jm2}) and (\ref{jm3})
from (\ref{jpm1}), (\ref{jpm2}) and (\ref{jpm3}).
\end{proof}

\bigskip

We next show that the polynomials $c_n^F$, $n\in \sigma_F$, (\ref{defchex}) and the corresponding difference operator $D_F$ (\ref{sodochex}) can be constructed by applying a sequence of at most $k$ Darboux transform to the Charlier system (where $k$ is the number of elements of $F$).

\begin{definition}\label{dxt}
Given a system $(T,(\phi_n)_n)$ formed by a second order difference operator $T$ and a sequence $(\phi_n)_n$ of eigenfunctions for $T$, $T(\phi_n)=\pi_n\phi_n$, by a Darboux transform of the system $(T,(\phi_n)_n)$ we
mean the following. For a real number $\lambda$, we factorize $T-\lambda Id$ as the product of two first order difference operators $T=BA+\lambda Id$ ($Id$ denotes the identity operator). We then produce a new system consisting in the operator $\hat T$, obtained by reversing the order of the factors,
$\hat T = AB+\lambda Id$, and the sequence of eigenfunctions $\hat \phi_n =A(\phi_n)$: $\hat T(\hat \phi_n)=\pi_n\hat\phi_n$.
We say that the system $(\hat T,(\hat\phi_n)_n)$ has been obtained by applying a Darboux transformation with parameter $\lambda$ to
 the system $(T,(\phi_n)_n)$.
\end{definition}

\begin{lemma}\label{lfe} Let $F=\{f_1,\cdots ,f_k\}$ be a finite set of positive integers and write $F_{\{ k\} }=\{f_1,\cdots ,f_{k-1}\}$ (see (\ref{deff0})). We define the first order difference operators $A_F$ and $B_F$ as
\begin{align}
A_F&=\frac{\Omega _F(x+1)}{\Omega_{F_{\{ k\} }}(x+1)}\Sh_0-\frac{\Omega _F(x)}{\Omega_{F_{\{ k\} }}(x+1)}\Sh_1,\\
B_F&=-x\frac{\Omega _{F_{\{ k\} }}(x+1)}{\Omega_{F}(x)}\Sh_ {-1}+a\frac{\Omega _{F_{\{ k\} }}(x)}{\Omega_{F}(x)}\Sh_0.
\end{align}
Then $c_n^F(x)=A_F(c_{n-f_k+k}^{F_{\{ k\} }})(x)$, $n\in
\sigma_F$. Moreover
\begin{align*}
D_{F_{\{ k\} }}&=B_FA_F+(f_k+u_{F_{\{ k\} }})Id,\\
D_{F}&=A_FB_F+(f_k+u_F)Id.
\end{align*}
In other words, the system
$(D_F,(c_n^F)_{n\in \sigma _F})$ can be obtained by applying a
Darboux transform to the system $(D_{F_{\{ k\} }},(c_n^{F_{\{ k\}
}})_{n\in \sigma _{F_{\{ k\} }}})$.
\end{lemma}
\begin{proof}
First of all, we point out that $\sigma_F=f_k-k+\sigma_{F_{\{ k\} }}$ (that is an easy consequence of (\ref{defuf}) and (\ref{defsf})).
In particular $u_F=u_{F_{\{ k\} }}+f_k-k$.

If we apply Sylvester's identity with $i_0=j_0=1$, $i_1=j_1=k$ (see Lemma \ref{lemS}) to the determinant (\ref{defchex}), we get
\begin{align*}
c_n^F(x)&=\frac{\Omega _F(x+1)}{\Omega_{F_{\{ k\} }}(x+1)}c_{n-f_k+k}^{F_{\{ k\} }}(x)-\frac{\Omega _F(x)}{\Omega_{F_{\{ k\} }}(x+1)}c_{n-f_k+k}^{F_{\{ k\} }}(x+1)\\
&=A_F(c_{n-f_k+k}^{F_{\{ k\} }})(x).
\end{align*}
Write now $D_{F_{\{ k\} }}=h_{-1}^{F_{\{ k\} }}\Sh_{-1}+h_0^{F_{\{ k\} }}\Sh_0+h_1^{F_{\{ k\} }}\Sh_1$. Using Lemma \ref{lemdes}, the factorization $D_{F_{\{ k\} }}=B_FA_F-(f_k+u_{F_{\{ k\} }})Id$ will follow if we prove
$$
h_0^{F_{\{ k\} }}(x)-(f_k+u_{F_{\{ k\} }})=-h_{-1}^{F_{\{ k\} }}(x)\frac{\Omega_F(x-1)}{\Omega_F(x)}-h_{1}^{F_{\{ k\} }}(x)\frac{\Omega_F(x+1)}{\Omega_F(x)}.
$$
This can be rewritten as
\begin{equation}\label{alqr}
D_{F_{\{ k\} }}(\Omega_F)=(f_k+u_{F_{\{ k\} }})\Omega _F.
\end{equation}
But this is a consequence of the identity $\Omega
_F(x)=(-1)^{k-1}c_{f_k+u_{F_{\{ k\} }}}^{F_{\{ k\} }}(x)$
(\ref{rrom0}).

We finally prove the factorization $D_{F}=A_FB_F-f_kId$. Since
$D_F(c_n^F)=nc_n^F$, $n\in \sigma_F$, using Lemma \ref{lemigop},
it will be enough to prove that $A_FB_F(c_n^F)=(n-f_k-u_F)c_n^F$,
$n\in \sigma_F$:
\begin{align*}
A_FB_F(c_n^F)&=A_FB_FA_F(c_{n-f_k+k}^{F_{\{ k\} }})=A_F[D^{F_{\{ k\} }}-(f_k+u_{F_{\{ k\} }})Id](c_{n-f_k+k}^{F_{\{ k\} }})\\
&=A_F[(n-f_k-u_F)(c_{n-f_k+k}^{F_{\{ k\} }})]=(n-f_k-u_F)c_n^F.
\end{align*}

\end{proof}

Analogous factorization can be obtained by using any of the sets $F_{\{i\} }$, $1\le i<k$ (see (\ref{deff0})) instead of $F_{\{k\} }$.

\bigskip

When the determinants $\Omega _F (n)\not =0$ (\ref{defom}), $n\ge 0$ (or equivalently, $\Phi_n^F\not =0$ (\ref{defphch}), $n\ge 0$),
the following alternative construction of the polynomial $q_n^F$ (\ref{defqnch}) has been given in \cite{DdI}.
For a finite set $F=\{f_1,\cdots , f_k\}$ of positive integers, consider the involuted set $I(F)=G=\{ g_1,\cdots, g_m\}$ with $g_i<g_{i+1}$, where the involution $I$ is defined by (\ref{dinv}). Assuming that $\Omega _F (n)\not =0$, $n\ge 0$, using the invariance (\ref{iza}) and Theorem 1.1 of \cite{DdI}, we have
\begin{equation}\label{quschi}
q_n^F(x)=\alpha_n\begin{vmatrix}
c^a_n(x-v_F) & -c^a_{n-1}(x-v_F) & \cdots & (-1)^mc^a_{n-m}(x-v_F) \\
c^{-a}_{g_1}(-n-1) & c^{-a}_{g_1}(-n) & \cdots &
c^{-a}_{g_1}(-n+m-1) \\
               \vdots & \vdots & \ddots & \vdots \\
               c^{-a}_{g_m}(-n-1) & \displaystyle
               c^{-a}_{g_m}(-n) & \cdots &c^{-a}_{g_m}(-n+m-1)
             \end{vmatrix},
\end{equation}
where $\alpha_n$, $n\ge 0$, is the normalization constant
$$
\alpha_n=(-1)^{k(n+1)}\frac{a^{k(n-1)-u_F}\prod_{f\in F}f!}{\prod_{i=1}^k(n+i)!}.
$$
The duality (\ref{duaqnrn}) then provides an alternative definition of the polynomial $c_n^F$, $n\ge v_F$. Indeed, after an easy calculation, we
conclude that
\begin{equation}\label{quschi2}
c_n^F(x)=\beta_n\begin{vmatrix}
c^a_{n-v_F}(x) & \frac{x}{a}c^a_{n-v_F}(x-1) & \cdots & \frac{(x-m+1)_m}{a^m}c^a_{n-v_F}(x-m+1) \\
c^{-a}_{g_1}(-x-1) & c^{-a}_{g_1}(-x) & \cdots &
c^{-a}_{g_1}(-x+m-1) \\
               \vdots & \vdots & \ddots & \vdots \\
               c^{-a}_{g_m}(-x-1) & \displaystyle
               c^{-a}_{g_m}(-x) & \cdots &c^{-a}_{g_m}(-x+m-1)
             \end{vmatrix},
\end{equation}
where $\beta_n$, $n\ge 0$, is the normalization constant
\begin{equation}\label{nc1}
\beta_n=(-1)^{m+k+u_F}\frac{a^m(n-v_F)!V_F\prod_{g\in G}g!\prod_{i=1}^k(f_i-n+u_F)}{(n-u_F)!V_G\prod_{f\in F}f!}.
\end{equation}
When the cardinality of the involuted set $G=I(F)$ is less than the cardinality of $F$, (\ref{quschi2}) will provide a more efficient way than (\ref{defchex}) for an explicit computation of the polynomials $c_n^F$, $n\ge v_F$. For instance, take $F=\{1,\cdots, k\}$.
Since $I(F)=\{ k\}$, the determinant in (\ref{quschi2}) has order $2$ while the determinant in (\ref{defchex}) has order $k+1$.

Applying Sylvester's identity to the determinant (\ref{quschi2}), we get an alternative way to construct the system $(D_F,c_n^F)$ by applying a sequence of at most $m$ Darboux transform to the Charlier system.

\begin{lemma}\label{thjod}
Given a real number $a\not =0$ and a finite set $F$of positive integers for which  $\Omega _F^a (n)\not =0$, $n\ge 0$, define the first order difference operators $C_F$ and $E_F$ as
\begin{align}\label{defoc}
C_F&=-\frac{x\tilde \Omega _F(x+1)}{a\tilde\Omega_{F_{\Downarrow }}(x)}\Sh_ {-1}+\frac{\tilde \Omega _F(x)}{\tilde \Omega_{F_{\Downarrow }}(x)}\Sh_0,\\\label{defoe}
E_F&=a\frac{\tilde \Omega _{F_{\Downarrow }}(x+1)}{\tilde \Omega_{F}(x+1)}\Sh_0-a\frac{\tilde \Omega_{F_{\Downarrow }}(x)}{\tilde \Omega _{F}(x+1)}\Sh_1,
\end{align}
where $F_{\Downarrow }$ is the finite set of positive integers defined by (\ref{deff1}).
Then $D_{F_{\Downarrow }}=E_FC_F+(u_{F}-k-1)Id$ and $D_{F}=C_FE_F+u_FId$. Moreover
\begin{equation}\label{spmdv}
C_F(c_{n-k-1}^{F_{\Downarrow }})=(-1)^{u_F+u_{F_{\Downarrow }}+1}\frac{n-u_F}{a}c_n^F(x),\quad n\ge v_F,
\end{equation}
where $n_{F_{\Downarrow }}$ is the number of elements of $F_{\Downarrow }$.
\end{lemma}

\begin{proof}

Write $D_{F_{\Downarrow }}=h_{-1}^{F_{\Downarrow }}\Sh_{-1}+h_0^{F_{\Downarrow }}\Sh_0+h_1^{F_{\Downarrow }}\Sh_1$. Using Lemma \ref{lemdes}, the factorization $D_{F_{\Downarrow }}=E_FC_F-(f_k+u_{F_0})Id$ will follow if we prove
\begin{equation}\label{alqr2}
h_0^{F_{\Downarrow }}(x)-(u_{F}-k-1))=-\frac{a}{x}h_{-1}^{F_{\Downarrow }}(x)\frac{\tilde \Omega_F(x)}{\tilde \Omega_F(x+1)}-\frac{x+1}{a}h_{1}^{F_{\Downarrow }}(x)\frac{\tilde \Omega_F(x+2)}{\tilde \Omega_F(x+1)}.
\end{equation}
If we set $a\to -a$, $x\to -x-1$ and use the invariant property of $\Omega$ (\ref{iza}), this can be rewritten as
$$
D_{G_{\{ m\}}}(\Omega_G)=(g_m+u_{G_{\{ m\}}})\Omega _G,
$$
where $G_{\{ m\} }$ is the finite set of positive integers defined by (\ref{deff0}).
(\ref{alqr2}) then follows by taking into account that
$\Omega _G(x)=(-1)^{m-1}c_{g_m+u_{G_{\{ m\}}}}^{G_{\{ m\}}}(x)$ (\ref{rrom0}).

For $n\ge v_F$, the identity (\ref{spmdv}) follows by applying Sylvester identity to the determinant (\ref{quschi2}) and using (\ref{ref0f1}).

The factorization $D_{F}=C_FE_F+u_FId$ can be proved as in Lemma \ref{lfe}.
\end{proof}

We have computational evidences which show that (\ref{spmdv}) also holds for $n\in \sigma_F$, $n<v_F$. Actually, in the next Section we will prove it for admissible sets $F$.

The factorization in the previous Lemma will be the key to prove that for admissible sets $F$, the polynomials $c_n^F$, $n\in\sigma _F$, are complete in the associated $L^2$ space.

\section{Exceptional Charlier polynomials}
In the previous Section, we have associated to each finite set $F$ of positive integers the polynomials $c_n^F$, $n\in \sigma_F$,
which are always eigenfunctions of a second order difference operator with rational coefficients.
We are interested in the cases when, in addition, those polynomials are orthogonal and complete with respect to a positive measure.

\begin{definition} The polynomials $c_n^{a;F}$, $n\in \sigma_F$, defined by (\ref{defchex}) are called exceptional Charlier polynomials, if they are orthogonal and complete with respect to a positive measure.
\end{definition}

We next introduce the key concept for finite sets $F$ such that the polynomials $c_n^F$, $n\in \sigma _F$, are exceptional Charlier polynomials.

\begin{definition}\label{defadm} Let $F$ be  a finite set  of positive integers. Split up the set $F$, $F=\bigcup _{i=1}^KY_i$, in such a way that $Y_i\cap Y_j=\emptyset $, $i\not =j$, the elements of each $Y_i$ are consecutive integers and $1+\max (Y_i)<\min Y_{i+1}$, $i=1,\cdots, K-1$. We say that $F$ is admissible if each $Y_i$, $i=1,\cdots, K$,  has an even number of elements.
\end{definition}

Admissible sets $F$ can be characterized in terms of the positivity of the measure $\rho_a^F$ (\ref{mraf}) and the sign of the Casorati polynomial $\Omega _F$ in $\NN$.

\begin{lemma}\label{l3.1} Given a positive real number $a$ and  a finite set $F$ of positive integers, the following conditions are equivalent.
\begin{enumerate}
\item The measure $\rho_a^F$ (\ref{mraf}) is positive.
\item The finite set $F$ is admissible.
\item $\Omega_F^a(n)\Omega_F^a(n+1)>0$ for all nonnegative integer $n$, where the polynomial $\Omega_F^a$ is defined by (\ref{defom}).
\end{enumerate}
\end{lemma}

\begin{proof}
It is clear that the definition of an admissible set $F$ is equivalent to $\prod_{f\in F}(x-f)\ge 0$, for all $x\in \NN$.
The equivalence between (1) and (2) is then an easy consequence of the definition of the measure $\rho _a^F$.

We now prove the equivalence between (1) and (3).

(1) $\Rightarrow$ (3). Since the measure $\rho_a^F$ is positive, the polynomials $(q_n^F)_n$ (\ref{defqnch}) are orthogonal with respect to the measure $\rho_a^F$ and have positive $L^2$-norm. According to (\ref{n2q}) in Lemma \ref{lemmc}, we have
\begin{equation}\label{nssu}
\langle q_n^F,q_n^F\rangle =(-1)^k\frac{n!}{(n+k)!}\langle c_n^a,c_n^a\rangle \Phi_n^F\Phi_{n+1}^F.
\end{equation}
We deduce then that $(-1)^k\Phi_n^F\Phi_{n+1}^F>0$ for all $n$. Using the duality (\ref{duomph}), we conclude that  $\Omega _F(n)\Omega_F(n+1)>0$ for all nonnegative integers $n$.

(3) $\Rightarrow$ (1). Using Lemma \ref{sze}, the duality (\ref{duomph}) and proceeding as before, we conclude that the polynomials $(q_n^F)_n$
are orthogonal with respect to $\rho_a^F$ and have positive $L^2$-norm. This implies that there exists a positive measure $\mu$ with respect to which the polynomials $(q_n^F)_n$ are orthogonal. Taking into account that the Fourier transform of $\rho_a^F$ is an entire function, using moment problem standard techniques (see, for instance, \cite{Akh}), it is not difficult to prove that $\mu$ has to be equal to $\rho _a^F$. Hence the measure $\rho_a^F$ is positive.
\end{proof}

In the two following Theorems we prove that for admisible sets $F$ the polynomials $c_n^F$, $n\in \sigma _F$, are orthogonal and complete with respect to a positive measure.

\begin{theorem}\label{th4.4} Given a real number $a\not =0$ and  a finite set $F$ of positive integers, assume that $\Omega_F^a(n)\not=0$ for all nonnegative integer $n$. Then the  polynomials $c_n^{a;F}$, $n\in \sigma _F$,
 are orthogonal with respect to the (possibly signed) discrete measure
\begin{equation}\label{mochex}
\omega_{a;F}=\sum_{x=0}^\infty \frac{a^x}{x!\Omega_F^a(x)\Omega_F^a(x+1)}\delta_x.
\end{equation}
Moreover, for $a<0$ the measure $\omega_{a;F}$ is never positive, and for $a>0$ the measure $\omega_{a;F}$ is positive if and only if $F$ is admissible.
\end{theorem}

\begin{proof}
Write $\Aa$ for the linear space generated by the polynomials $c_n^F$, $n\in \sigma _F$.
Using Lemma \ref{tcsd}, the definition of the measure $\omega_{a;F}$ and the expressions for the difference coefficients of the operator $D_F$ (see Theorem \ref{th3.3}), it is straightforward to check that $D_F$  is symmetric with respect to the pair $(\omega_{a;F},\Aa )$.
Since the polynomials $c_n^F$, $n\in \sigma_F$, are eigenfunctions
of $D_F$ with different eigenvalues, Lemma \ref{lsyo} implies that they are orthogonal with respect to $\omega_{a;F}$.

If $a<0$ and the measure $\omega_{a;F}$ is positive, we conclude that $\Omega_F(2n+1)\Omega_F(2n+2)<0$ for all positive integer $n$. But this would imply that $\Omega_F$ has at least a zero in each interval $(2n+1,2n+2)$, which it is impossible since $\Omega_F$ is a polynomial.

If $a>0$, according to Lemma \ref{l3.1}, $F$ is admissible if and only if $\Omega_F(x)\Omega_F(x+1)>0$ for all nonnegative integer $x$.

\end{proof}

\begin{theorem}\label{th4.5} Let $a$ and $F$ be a positive real number and an admissible finite set of positive integers, respectively. Then the linear combinations of the  polynomials $c_n^{a;F}$, $n\in \sigma _F$, are dense in $L^2(\omega_{a;F})$, where $\omega_{a;F}$ is the positive measure (\ref{mochex}). Hence $c_n^{a;F}$, $n\in \sigma _F$, are exceptional Charlier polynomials.
\end{theorem}

\begin{proof}
Using Lemma \ref{l3.1} and taking into account that $F$ is admissible, it follows that the measure $\rho _{a}^{F}$ (\ref{mraf}) is positive. We remark that this positive measure is also determinate (that is, there is not other measure with the same moments as those of $\rho _{a}^{F}$). As we pointed out above, this can be proved using moment problem standard techniques (taking into account, for instance, that the Fourier transform of $\rho_{a}^F$ is an entire function). Since for determinate measures the polynomials are dense in the associated $L^2$ space, we deduce that the sequence $(q_n^F/\Vert q_n^F\Vert _2)_n$ (where $q_n^F$ is the polynomial defined by (\ref{defqnch})) is an orthonormal basis in $L^2(\rho_{a}^F)$.

For $s\in \sigma _F$, consider the function $h_s(x)=\begin{cases} 1/\rho _{a}^{F}(s),& x=s\\ 0,& x\not =s, \end{cases}$ where by $\rho _{a}^{F}(s)$ we denote the mass  of the discrete measure $\rho_a^F$ at the point $s$.
Since the support of $\rho _{a}^{F}$ is $\sigma_F$, we get that $h_s\in L^2(\rho_{a}^F)$. Its Fourier coefficients with respect to the orthonormal basis $(q_n^F/\Vert q_n^F\Vert _2)_n$ are $q_n^F(s)/\Vert q_n^F\Vert _2$, $n\ge 0$. Hence
\begin{equation}\label{pf1}
\sum _{n=0}^\infty \frac{q_n^F(s)q_n^F(r)}{\Vert q_n^F\Vert _2 ^2}=\langle h_s,h_r\rangle _{\rho_{a}^F}=\frac{1}{\rho_{a}^F(s)}\delta_{s,r}.
\end{equation}
This is the dual orthogonality associated to the orthogonality
$$
\sum_{u\in \sigma _F}q_n^F(u)q_m^F(u)\rho _{a}^{F}(u)=\langle q_n^F,q_n^F\rangle \delta _{n,m}
$$
of the polynomials $q_n^F$, $n\ge 0$, with respect to the positive measure $\rho _{a}^{F}$ (see, for instance, \cite{At}, Appendix III, or \cite{KLS}, Th. 3.8).

Using (\ref{nssu}), (\ref{norCh}) and the duality (\ref{duomph}), we get
\begin{equation}\label{neq1}
\frac{1}{\Vert q_n^F\Vert _2 ^2}=\omega _{a;F}(n)x_n,
\end{equation}
where $x_n$ is the positive number given by
\begin{equation}\label{defxn}
x_n=\frac{a^k}{e^a}\left(\frac{\prod_{i=0}^k (n+i)!}{a^{(k+1)n-u_F}\prod_{f\in F}f!}\right)^2
\end{equation}
Using now the duality (\ref{duaqnrn}), we can rewrite (\ref{pf1}) for $n=m$ as
\begin{equation}\label{mochx}
\langle c_n^{a;F},c_n^{a;F}\rangle_{\omega_{a}^{F}}=\frac{a^{n-u_F-k}e^a\prod_{f\in
F}(n-f-u_F)}{(n-u_F)!}.
\end{equation}

Consider now a function $f$ in $L^2(\omega_{a;F})$ and write $g(n)=(-1)^nf(n)/x_n^{1/2}$, where $x_n$ is the positive number given by (\ref{defxn}). Using (\ref{neq1}), we get
$$
\sum_{n=0}^\infty\frac{\vert g(n)\vert ^2}{\langle q_n^F,q_n^F\rangle _{\rho_{a}^F}}=\sum_{n=0}^\infty \omega_{a;F}(n)\vert f(n)\vert ^2=\Vert f\Vert _2^2<\infty.
$$
Define now
$$
v_r=\sum_{n=0}^\infty\frac{g(n)q_n^F (r)}{\langle q_n^F,q_n^F\rangle _{\rho_{a}^F}}.
$$
Using Theorem III.2.1 of \cite{At}, we get
\begin{equation}\label{pf6}
\Vert f\Vert _2^2=\sum_{n=0}^\infty\frac{\vert g(n)\vert ^2}{\langle q_n^F,q_n^F\rangle _{\rho_{a}^F}}=\sum _{r\in \sigma _F}\vert v_r\vert ^2\rho_{a}^F (r).
\end{equation}
On the other hand, using the duality (\ref{duaqnrn}),  (\ref{neq1}), (\ref{defxn}) and (\ref{mochx}), we have
$$
v_r=\frac{(-1)^r}{(\rho_{a}^F(r))^{1/2}}\sum_{n=0}^\infty f(n)\frac{c_r^{a;F}(n)}{\Vert c_r^{a;F}\Vert _2}\omega_{a;F}(n).
$$
This is saying that $(-1)^r(\rho_{a}^F(r))^{1/2}v_r$, $r\in \sigma _\F$, are the Fourier coefficients of $f$ with respect to the orthonormal system $(c_n^{a;F}/\Vert c_n^{a;F}\Vert_2)_n$. Hence, the identity (\ref{pf6}) is Parseval's identity for the function $f$. From where we deduce that the orthonormal system $(c_n^{a;F}/\Vert c_n^{a;F}\Vert_2)_n$ is complete in $L^2(\omega_{a;F})$.
\end{proof}

\section{Constructing polynomials which are eigenfunctions of second order differential operators}
One can construct exceptional Hermite polynomials by taking limit in the exceptional Charlier polynomials. We use the basic limit
(\ref{blchh}).

Given a finite set of positive integers $F$, using the expression (\ref{defchexa}) for the polynomials $c_n^{a;F}$, $n\in\sigma_F$, setting $x\to \sqrt {2a}x+a$ and taking limit as $a\to +\infty$, we get (up to normalization constants) the polynomials, $n\in \sigma _F$,
\begin{equation}\label{defhex}
H_n^F(x)=\begin{vmatrix}H_{n-u_F}(x)&H_{n-u_F}'(x)&\cdots &H_{n-u_F}^{(k)}(x)\\
H_{f_1}(x)&H_{f_1}'(x)&\cdots &H_{f_1}^{(k)}(x)\\
\vdots&\vdots&\ddots &\vdots\\
H_{f_k}(x)&H_{f_k}'(x)&\cdots &H_{f_k}^{(k)}(x) \end{vmatrix} .
\end{equation}
More precisely
\begin{equation}\label{lim1}
\lim_{a\to +\infty}\left(\frac{2}{a}\right)^{n/2}c_n^{F}(\sqrt{2a}x+a)=\frac{1}{(n-u_F)!\nu_F}H_n^F(x)
\end{equation}
uniformly in compact sets, where
\begin{equation}\label{defnuf}
\nu_F=2^{\binom{k+1}{2}}\prod_{f\in F}f!.
\end{equation}

Notice that $H_n^F$ is a polynomial of degree $n$ with leading coefficient equal to
$$
2^{n+\binom{k+1}{2}}V_F\prod_{f\in F}(f-n+u_F),
$$
where $V_F$ is the Vandermonde determinant defined by (\ref{defvdm}).

Assume now that $F$ is admissible (\ref{defadm}). According to Lemma \ref{l3.1}, this gives for all $a>0$ that $\Omega ^a _F(x)\Omega ^a _F(x+1)>0$ for $x\in \NN $, where $\Omega _F^a$ is the polynomial (\ref{defom}) associated to the Charlier family. In particular $\Omega ^a _F(x)\not =0$, for all nonnegative integer $x$. Hence, if instead of (\ref{defchexa}) we use (\ref{quschi2}), we get the following alternative expression for the polynomials $H_n^F$, $n\ge v_F$, ($i$ denotes the imaginary unit $i=\sqrt{-1}$)
\begin{equation}\label{defhexa}
H_n^F(x)=\gamma_n\begin{vmatrix}H_{n-v_F}(x)&-iH_{n-v_F+1}(x)&\cdots &(-i)^mH_{n-v_F+m}(x)\\
H_{g_1}(-ix)&H_{g_1}'(-ix)&\cdots &H_{g_1}^{(m)}(-ix)\\
\vdots&\vdots&\ddots &\vdots\\
H_{g_m}(-ix)&H_{g_m}'(-ix)&\cdots &H_{g_m}^{(m)}(-ix)\end{vmatrix} ,
\end{equation}
where $\gamma_n$ is the normalization constant
\begin{equation}\label{nc2}
\gamma_n=i^{u_G}2^{\binom{k+1}{2}-\binom{m}{2}}\frac{V_F}{V_G}\prod_{f\in F}(f-n+u_F),
\end{equation}
and as in the previous sections $G$ denotes the involuted set $G=I(F)$ (see (\ref{dinv})).

We introduce the associated polynomials
\begin{align}\label{defhom}
\Omega _F(x)&=\begin{vmatrix}
H_{f_1}(x)&H_{f_1}'(x)&\cdots &H_{f_1}^{(k-1)}(x)\\
\vdots&\vdots&\ddots &\vdots\\
H_{f_k}(x)&H_{f_k}'(x)&\cdots &H_{f_k}^{(k-1)}(x) \end{vmatrix},\\\label{defhomt}
\tilde \Omega _F(x)&=i^{u_G+m}\begin{vmatrix}
H_{g_1}(-ix)&H_{g_1}'(-ix)&\cdots &H_{g_1}^{(m-1)}(-ix)\\
\vdots&\vdots&\ddots &\vdots\\
H_{g_m}(-ix)&H_{g_m}'(-ix)&\cdots &H_{g_m}^{(m-1)}(-ix)\end{vmatrix} .
\end{align}
Since $u_G+m=u_F+k$, we have that both $\Omega_F$ and $\tilde \Omega_F$ are polynomials of degree $u_F+k$.

The invariant property (\ref{iza}) gives
\begin{equation}\label{izah}
\Omega _F(x)=2^{\binom{k}{2}-\binom{m}{2}}\frac{V_F}{V_G}\tilde\Omega _F(x).
\end{equation}
We also straightforwardly have
\begin{equation}\label{rromh}
H_{u_F}^F(x)=\frac{2^{k-s_F+1}\nu_F}{\nu_{F_\Downarrow}}\Omega_{F_\Downarrow }(x),
\end{equation}
where the numbers $\nu_F$ and $s_F$ are defined by (\ref{defnuf}) and (\ref{defs0}), respectively, and the finite set of integers $F_\Downarrow$ is defined by (\ref{deff1}).

Proceeding in a similar way, we can transform the second order difference operator (\ref{sodochex}) in a second order differential operator with respect to which the polynomials $H_n^F$, $n\in\sigma_F$, are eigenfunctions:

\begin{theorem}\label{th5.1} Let $F$ be a finite set of positive integers. Then the polynomials $H_n^F$, $n\in \sigma _F$,
are common eigenfunctions of the second order differential operator
\begin{equation}\label{sodohex}
D_F=-\partial ^2+h_1(x)\partial+h_0(x),
\end{equation}
where $\partial=d/dx$ and
\begin{align}\label{jph1}
h_1(x)&=2\left(x+\frac{\Omega_F'(x)}{\Omega_F(x)}\right),\\\label{jph2}
h_0(x)&=2\left(k+u_F-x\frac{\Omega_F'(x)}{\Omega_F(x)}\right)-\frac{\Omega_F''(x)}{\Omega_F(x)}.
\end{align}
More precisely $D_F(H_n^F)=2nH_n^F(x)$.
\end{theorem}

\begin{proof}
The proof is a matter of calculation using carefully the basic limit (\ref{blchh}), hence we only sketch it.

We assume that $k$ is even (the case for $k$ odd being similar).

Using that $c_n^a(x+k)=\sum_{j=0}^k\binom{k}{j}c_{n-j}^a(x)$, the basic limit
(\ref{blchh}) and the alternative definition (\ref{defoma}) for $\Omega _F^a$, we can get the limits
\begin{align}\label{lim2}
\lim_{a\to \infty}\left(\frac{2}{a}\right)^{(u_F+k)/2}\Omega ^a_F(x_a)&=\frac{2^k\Omega _F(x)}{\nu _F},\\\label{lim3}
\lim_{a\to \infty}\left(\frac{2}{a}\right)^{(u_F+k-1)/2}(\Omega ^a_F(x_a+1)-\Omega ^a_F(x_a))&=\frac{2^{k-1}\Omega _F'(x)}{\nu _F},\\\nonumber
\lim_{a\to \infty}\left(\frac{2}{a}\right)^{(u_F+k-2)/2}(\Omega ^a_F(x_a+1)-2\Omega ^a_F(x_a)+\Omega ^a_F(x_a-1))&=\frac{2^{k-2}\Omega _F''(x)}{\nu_F},
\end{align}
where $\nu_F$ is defined by (\ref{defnuf}) and $x_a=\sqrt{2a}x+a$.

Taking into account that $c_n^{a;F}(x)=\Omega_{F_n}^a(x)$, where $F_n=\{f_1,\cdots, f_k,n-u_F\}$, we can get similar limits for the polynomials
$c_n^{a;F}(x)$, $n\in \sigma _F$.

We next write the spectral equation $D_F^a(c_n^{a;F})=nc_n^{a;F}$ (where we write $D_F^a$ for the second order difference operator (\ref{sodochex})) in the form
\begin{align*}
&h_{-1}^a(x)\left[c_n^F(x+1)-2c_n^F(x)+c_n^F(x-1)\right]+(h_1^a(x)-h^a_{-1}(x))\left[c_n^F(x+1)-c_n^F(x)\right]\\&\quad\quad +(h_0^a(x)+h_1^a(x)+h^a_{-1}(x))c_n^F(x)=nc_n^F(x),
\end{align*}
where $h_{-1}^a, h_0^a$ and $h_1^a$ are given by (\ref{jpm1}), (\ref{jpm2}) and (\ref{jpm3}), respectively.
It is then enough to set $x\to x_a$ and take carefully limit as $a\to \infty$ using (\ref{jpm1}), (\ref{jpm2}), (\ref{jpm3}) and the previous limits.

\end{proof}

We can factorize the second order differential operator $D_F$ as product of two first order differential operators. As a consequence the system $(D_F, (H_n^F)_{n\in \sigma_F})$ can be constructed by applying a sequence of $k$ Darboux transforms to the Hermite system.

\begin{lemma} Let $F=\{f_1,\cdots ,f_k\}$ be a finite set of positive integers and write $F_{\{ k\}}=\{f_1,\cdots ,f_{k-1}\}$. We define the first order differential operators $A_F$ and $B_F$ as
\begin{align}
A_F&=-\frac{\Omega _F(x)}{\Omega_{F_{\{ k\}}}(x)}\partial+\frac{\Omega _F'(x)}{\Omega_{F_{\{ k\}}}(x)},\\
B_F&=\frac{\Omega _{F_{\{ k\}}}(x)}{\Omega_{F}(x)}\partial-\frac{2x\Omega_{F_{k}}+\Omega '_{F_{\{ k\}}}(x)}{\Omega_{F}(x)}.
\end{align}
Then $H_n^F(x)=A_F(H_{n-f_k+k}^{F_{\{ k\}}})(x)$, $n\in \sigma_F$. Moreover
\begin{align*}
D_{F_{\{ k\}}}&=B_FA_F+2(f_k+u_{F_{\{ k\}}})Id,\\
D_{F}&=A_FB_F+2(f_k+u_F)Id.
\end{align*}
\end{lemma}
\begin{proof}
The Lemma can be proved applying limits in Lemma \ref{lfe}, or by applying Silvester Identity (for rows $(1,k)$ and columns $(k-1,k)$)
in the definition (\ref{defhex}) of the polynomials $H_n^F$, $n\in \sigma _F$.

\end{proof}

When $F$ is admissible, using the alternative expression (\ref{defhexa}) for the polynomials $H_n^F$, $n\in \sigma _F$, we get other factorization for the differential operator
$D_F$.

\begin{lemma}\label{lfh} Let $F$ be an admissible finite set of positive integers and write $F_{\Downarrow}$ for the finite set of positive integers defined by (\ref{deff1}). We define the first order differential operators $C_F$ and $E_F$ as
\begin{align}\label{opch}
C_F&=\frac{\tilde \Omega _F(x)}{\tilde \Omega_{F_{\Downarrow}}(x)}\partial-\frac{\tilde \Omega _F'(x)+2x\tilde\Omega _F(x)}{\tilde \Omega_{F_{\Downarrow}}(x)},\\\label{opeh}
E_F&=-\frac{\tilde \Omega _{F_{\Downarrow}}(x)}{\tilde \Omega_{F}(x)}\partial+\frac{\tilde \Omega '_{F_{\Downarrow}}(x)}{\tilde \Omega_{F}(x)}.
\end{align}
Then $D_{F_{\Downarrow}}=E_FC_F+2(u_{F}-k-1)Id$ and $D_{F}=C_FE_F+2u_FId$. Moreover
\begin{equation}\label{spmdvh}
C_F(H_{n-k-1}^{F_{\Downarrow}})=\frac{-2^{m+\binom{k-s_F+2}{2}-\binom{k+1}{2}-1}\prod_{j=1}^{m-1}(g_m-g_j)}{\prod_{j=1}^{s_F-1}(j-1)!(j-n+u_F)\prod_{f\in F;f>s_F}(f-j)}H_n^F(x),\quad n\ge v_F,
\end{equation}
where $G=I(F)=\{g_1,\ldots , g_m\}$ and $s_F$ is defined by (\ref{defs0}).
\end{lemma}

\section{Exceptional Hermite polynomials}
In the previous Section, we have associated to each finite set $F$ of positive integers the polynomials $H_n^F$, $n\in \sigma_F$,
which are always eigenfunctions of a second order differential operator with rational coefficients.
We are interested in the cases when, in addition, those polynomials are orthogonal and complete with respect to a positive measure.

\begin{definition} The polynomials $H_n^F$, $n\in \sigma_F$, defined by (\ref{defhex}) are called exceptional Hermite polynomials, if they are orthogonal and complete with respect to a positive measure.
\end{definition}

As it was mentioned in the Introduction, simultaneously with this paper, exceptional Hermite polynomials as Wronskian determinant of Hermite polynomials have been introduced and studied (using a different approach) in \cite{GUGM}. In that paper, exceptional Hermite polynomials are defined for a given non-decreasing finite sequence of non-negative integers
$\lambda=(\lambda _1,\cdots , \lambda _l)$, and are denoted by $H^{(\lambda )} _j(x)$; the degree of $H^{(\lambda )} _j(x)$ is $2\sum_{j=1}^l\lambda_j -2l+j$. The relationship between the exceptional Hermite polynomials introduced in \cite{GUGM} and the ones in this paper is the following: given a non-decreasing finite sequence of positive integers
$\lambda=(\lambda _1,\cdots , \lambda _l)$, we form a finite set of positive integers $F$ as follows: $F=\{f_1,f_2,\cdots ,f_{2l-1},f_{2l}\}$, where
$f_{2j-1}=\lambda_{j}+2j-2$ and $f_{2j}=\lambda_{j}+2j-1$, $j=1,\cdots , l$; then $H^{(\lambda )} _j(x)=H^F_{2\sum_{j=1}^l\lambda_j -2l+j}(x)$.

The following Lemma and Theorem show that again the admissibility of $F$ will be the key to construct exceptional Hermite polynomials.

\begin{lemma} Let $F$  be a finite set of positive integers. Then $F$ is admissible if and only if
the Wronskian determinant $\Omega_F$ (\ref{defhom}) does not vanish in $\RR$.
\end{lemma}
\begin{proof}
Consider a second order differential operator $T$ of the form $T=-d^2/dx^2+U$, and write $\phi_n$, $n\ge 0$, for a sequence of eigenfunctions for $T$.
For a finite set of positive integers $F=\{f_1,\cdots, f_k\}$, consider the Wronskian determinant $\Omega_F^T(x)=|\phi_{f_l}^{(j-1)}(x)|_{l,j=1}^k$.
For operators defined in a half-line, Krein proved \cite{Kr} that $F$ is admissible if and only if $\Omega _F^T$ does not vanish in the real line. A similar result was proved by Adler \cite{Ad} for operators defined in a bounded interval. Adler's result can easily  be extended to the whole real line (in fact, he considered in \cite{Ad} the case of Wronskian determinant of Hermite polynomials). The Lemma is then an easy consequence of this result for  the functions $H_n(x)e^{-x^2}$.

Anyway, for the sake of completeness, we prove by passing to limit from Lemma \ref{l3.1} the implication $\Rightarrow$ in the Lemma (which it is what we need in the following Theorem).

For $a>0$, consider the positive measure $\tau _a$ defined by
$$
\tau _a=\frac{a^k}{e^a}\sum _{x=0}^\infty\frac{a^x(c_{u_F}^{a;F}(x))^2}{x!\Omega _F^a(x)\Omega_F^a(x+1)}\delta {y_{a,x}},
$$
where
\begin{equation}\label{defya}
y_{a,x}=(x-a)/\sqrt{2a}.
\end{equation}
We also need the following limits
\begin{align}\label{lm1}
\lim _{a\to +\infty}\frac{\Omega _F^a(\sqrt {2a}x+a)}{a^{(k+u_F)/2}}&=\frac{2^{(k-u_F)/2}\Omega _F(x)}{\nu_F},\\\label{lm11}
\lim _{a\to +\infty}\frac{\Omega _F^a(\sqrt {2a}x+a+1)}{a^{(k+u_F)/2}}&=\frac{2^{(k-u_F)/2}\Omega _F(x)}{\nu_F},\\\label{lm2}
\lim _{a\to +\infty}\frac{c_{u_F}^{a;F}(\sqrt {2a}x+a)}{a^{u_F/2}}&=\frac{2^{k-u_F/2-s_F+1}\Omega _{F_{\Downarrow}}(x)}{\nu_{F_{\Downarrow}}},\\\label{lm3}
\lim _{a\to +\infty}\frac{\sqrt{2a}a^{\sqrt {2a}x+a}}{e^a\Gamma (\sqrt {2a}x+a+1)}&=e^{-x^2}/\sqrt \pi,
\end{align}
uniformly in compact sets. The  first limit is (\ref{lim2}). The second one is a consequence of
(\ref{lim3}). The third one is a consequence of (\ref{lim1}) and (\ref{rromh}). The forth one is consequence of Stirling's formula.

We proceed by complete induction on $s=\max F$.
Since $F$ is admissible, the first case to be considered is $s=2$ which it corresponds with $F=\{1,2\}$. Then
$\Omega _F (x)=8x^2+4$ which it clearly satisfies $\Omega _F(x)\not =0$, $x\in \RR $.

Asume that  $\Omega _F(x)\not =0$, $x\in \RR $, if $\max F\le s$  and take an admissible set $F$ with $\max F=s+1$. The definition of $F_{\Downarrow }$ (\ref{deff1}) then says that $\max F_{\Downarrow }\le s$. The induction hypothesis then implies that $\Omega _{F_{\Downarrow }}(x)\not =0$, $x\in \RR $.
We now proceed by \textsl{reductio ad absurdum}. Hence, we assume that the polynomial $\Omega _F$ vanishes in $\RR$. Write $x_0=\max \{x\in \RR: \Omega _F(x)=0\}$. Take real numbers $u,v$ with $x_0<u<v$ and write $I=[u,v]$. Since $\Omega_F(x)\not =0$, $x\in I$, applying Hurwitz's Theorem to the limits (\ref{lm1}) and (\ref{lm11}) we can choice a contable set $X=\{a_n: n\in \NN \}$ of positive numbers with $\lim_n a_n=+\infty$ such that $\Omega _F^a(\sqrt {2a}x+a)\Omega _F^a(\sqrt {2a}x+a+1)\not =0$, $x\in I$ and $a\in X$.

Hence, we can combine the limits (\ref{lm1}), (\ref{lm11}), (\ref{lm2}) and (\ref{lm3}) to get
\begin{equation}\label{lm4}
\lim _{a\to +\infty;a\in X}h_a(x)=d_3h(x),\quad \mbox{uniformly in $I$},
\end{equation}
where
\begin{align*}
h_a(x)&=\frac{a^k\sqrt{2a}a^{\sqrt {2a}x+a}(c_{u_F}^{a;F}(\sqrt {2a}x+a))^2}{e^a\Gamma (\sqrt {2a}x+a+1)\Omega_F^a (\sqrt {2a}x+a)\Omega_F^a (\sqrt {2a}x+a+1)},\\
h(x)&=\frac{e^{-x^2}\Omega ^2_{F_{\Downarrow }}(x)}{\Omega ^2_F(x)},
\end{align*}
and $d_3=2^{k-2s_F+2}\nu_F^2/(\sqrt \pi \nu_{F_{\Downarrow }}^2)$.
We now prove that
\begin{equation}\label{lm5}
\lim _{a\to +\infty ;a\in X}\tau _a(I)=d_3\int_{I}h(x)dx.
\end{equation}
To do that, write $I_a=\{ x\in \NN: a+u\sqrt{2a}\le x\le a+v\sqrt{2a}\}$.
The numbers $y_{a,x}$, $x\in I_a$, form a partition of the interval $I$ with $y_{a,x+1}-y_{a,x}=1/\sqrt{2a}$ (see (\ref{defya})). Since the function $h$ is continuous  in the interval $I$, we get that
$$
\int_{I}h(x)dx=\lim_{a\to +\infty; a\in X}S_a,
$$
where $S_a$ is the Cauchy sum
$$
S_a=\sum_{x\in I_a}h(y_{a,x})(y_{a,x+1}-y_{a,x}).
$$
On the other hand, since $x\in I_a$ if and only if $u\le y_{a,x}\le v$ (\ref{defya}), we get
\begin{align*}
\tau _a(I)&=\frac{a^k}{e^a}\sum _{x\in I_a}\frac{a^x(c_{u_F}^F(x))^2}{x!\Omega _F^a(x)\Omega_F^a(x+1)}=\frac{1}{\sqrt {2a}}\sum _{x\in I_a}h_a(y_{a,x})\\
&=\sum _{x\in I_a}h_a(y_{a,x})(y_{a,x+1}-y_{a,x}).
\end{align*}
The limit (\ref{lm5}) now follows from the uniform limit (\ref{lm4}).

The identity (\ref{mochx}) for $n=u_F$ says that $\tau _a(\RR)=d_F$, where the positive constant $d_F=\prod_{f\in F}f$ does not depend on $a$.
This gives $\tau _a(I)\le d_F$. And so from the limit (\ref{lm5}) we get
$$
\int_{I}h(x)dx \le \frac{d_F}{d_3}.
$$
That is
$$
\int _u^v\frac{e^{-x^2}\Omega ^2_{F_{\Downarrow }}(x)}{\Omega ^2_F(x)}dx\le \frac{d_F}{d_3}.
$$
On the other hand, since $\Omega_F(x_0)=0$ and $\Omega _{F_{\Downarrow }}(x)\not =0$, $x\in \RR$, we get
$$
\lim _{u\to x_0^+}\int _u^v\frac{e^{-x^2}\Omega ^2_{F_{\Downarrow }}(x)}{\Omega ^2_F(x)}dx=\infty.
$$
Which it is a contradiction.

\end{proof}

\begin{corollary} Given an admissible finite set $F$ of positive integers, we have
for $n\in \sigma _F$,
\begin{equation}\label{mohx}
\langle H_n^{F},H_n^{F}\rangle_{\omega _{F}}=\sqrt \pi 2^{n-u_F+k}(n-u_F)!\prod_{f\in
F}(n-f-u_F).
\end{equation}
\end{corollary}

\begin{proof}
The proof is similar to that of the previous Theorem (using (\ref{mochx})) and it is omitted.
\end{proof}

\begin{theorem}\label{th6.3} Let $F$ be an admissible finite set of positive integers. Then the polynomials $H_n^F$, $n\in \sigma _F$,
 are orthogonal with respect to the positive weight
\begin{equation}\label{mohex}
\omega_F(x)=\frac{e^{-x^2}}{\Omega ^2_F(x)},\quad x\in \RR,
\end{equation}
and their linear combinations  are dense in  $L^2(\omega_{F})$. Hence $H_n^F$, $n\in \sigma _F$, are
exceptional Hermite polynomials.
\end{theorem}

\begin{proof}
Write $\Aa _F$ for the linear space generated by the polynomials $H_n^F$, $n\in \sigma _F$.
Using Lemma \ref{tcsd2}, it is easy to check that the second order differential operator $D_F$ (\ref{sodohex}) is symmetric with respect to the pair $(\omega _F, \Aa _F)$ (\ref{mohex}). Since the polynomials $H_n^F$, $n\in \sigma _F$, are eigenfunctions of $D_F$ with different eigenvalues Lemma
\ref{lsyo} implies that they are orthogonal with respect to $\omega _F$.

The completeness of $H_n^F$, $n\in \sigma _F$, in  $L^2(\omega_{F})$ can be proved in a similar way to that of Proposition 5.8 in \cite{GUGM} and it is omitted.

\end{proof}

\bigskip
\noindent
\textit{Acknowledgement} The author would like to thank to two anonymous referees for their comments and suggestions.

\noindent
\textit{Mathematics Subject Classification: 42C05, 33C45, 33E30}

\noindent
\textit{Key words and phrases}: Orthogonal polynomials. Exceptional orthogonal polynomial. Difference operators. Differential operators.
Charlier polynomials. Hermite polynomials.

     \end{document}